\documentclass[12pt]{article}

\usepackage[russian,ukrainian,english]{babel}

\usepackage{amssymb}
\usepackage{graphics}

\newcommand{\qed}{$\;\;\;\Box$}
\newenvironment{proof}{\par\smallbreak{\sl Proof.~}}
{\unskip\nobreak\hfill \qed \par\medbreak}
\newenvironment{proofthm}{\par\smallbreak{\sl Proof of  Theorem~}}
{\unskip\nobreak\hfill \qed \par\medbreak}
\newcounter{claim}
\renewcommand{\theclaim}{\arabic{claim}}
{\par\medskip\par}

{\qed\par\smallbreak}
\newcommand{\hide}[1]{}

\newcommand{\D}{{\cal D}}

\newcommand{\N}{{\mathbb N}}

\newcommand{\R}{{\mathbb R}}

\newcommand{\M}{{\mathbb M}}

\newcommand{\A}{{\cal A}}

\newcommand{\LL}{{\cal L}}

\newcommand{\beq}{\begin{equation}}
\newcommand{\ee}{\end{equation}}

\renewcommand{\d}{\partial}

\newtheorem{thm}{Theorem}[section]
\newtheorem{lemma}[thm]{Lemma}

\newtheorem{defn}[thm]{Definition}

\newtheorem{ex}[thm]{Example}

\newcommand{\al}{\alpha}
\newcommand{\be}{\beta}
\newcommand{\ga}{\gamma}
\newcommand{\de}{\delta}
\newcommand{\eps}{\varepsilon}
\newcommand{\vphi}{\varphi}
\newcommand{\la}{\lambda}
\newcommand{\om}{\omega}
\newcommand{\io}{\iota}
\newcommand{\reff}[1]{(\ref{#1})}

\newcommand{\diag}{\mathop{\rm diag}\nolimits}

\newcommand{\tr}{\mathop{\rm tr}}

\renewcommand{\Im}{\mathop{\mathrm{Im}}\nolimits}

\date{}
\title{
Classical Bounded and Almost Periodic Solutions
to  \\ Quasilinear First-Order Hyperbolic Systems in a Strip
}

\newcounter{thesame}
\author{
I.~Kmit
\thanks{Institute of Mathematics, Humboldt University of Berlin. On leave from the
Institute for Applied Problems of Mechanics and Mathematics,
Ukrainian National Academy of Sciences. {\small   E-mail:
{\tt kmit@math.hu-berlin.de}}}
\ \ \ 
L.~Recke
\thanks{Institute of Mathematics, Humboldt University of Berlin. {\small   E-mail:
{\tt recke@math.hu-berlin.de}}}
\ \ \ 
V.~Tkachenko
\thanks{Institute of Mathematics, National Academy of Sciences of Ukraine. {\small   E-mail:
{\tt vitk@imath.kiev.ua}}}
}

\begin{document}

\maketitle

\begin{abstract}
\noindent
We consider   boundary value problems
for quasilinear first-order one-dimensional hyperbolic systems in a strip.
The  boundary conditions are supposed to be of a smoothing type, in the sense that
the  $L^2$-generalized solutions  to  the initial-boundary value  problems  become
eventually $C^2$-smooth for any initial $L^2$-data.
We investigate  small  global   classical solutions and
 obtain the existence and uniqueness result under the condition that the evolution family
generated by the linearized problem has  exponential dichotomy on $\R$. 
We prove
 that the dichotomy survives  under small perturbations in the leading coefficients  of the
hyperbolic system.
Assuming that  the coefficients of the hyperbolic system are almost periodic,  we prove that
the bounded solution is almost periodic also.
\end{abstract}



\section{Introduction}\label{sec:intr}
\renewcommand{\theequation}{{\thesection}.\arabic{equation}}
\setcounter{equation}{0}

\subsection{Problem setting and main result} 
\renewcommand{\theequation}{{\thesection}.\arabic{equation}}
\setcounter{equation}{0}

We consider
 first-order quasilinear hyperbolic systems   of the following type
\begin{equation}\label{eq:1}
\partial_t u + A(x,t,u)\partial_x u + B(x,t, u)u = f(x,t),\quad x\in (0,1),
\end{equation}
subjected to the (nonlocal) reflection boundary conditions
\beq\label{eq:3}
\begin{array}{l}
\displaystyle
u_j(0,t) = \sum\limits_{k=m+1}^np_{jk}u_k(0,t)+\sum\limits_{k=1}^mp_{jk}u_k(1,t),\quad  1\le j\le m,
\nonumber\\
\displaystyle
u_j(1,t) = \sum\limits_{k=m+1}^np_{jk}u_k(0,t)+ \sum\limits_{k=1}^mp_{jk}u_k(1,t), \quad   m< j\le n,
\nonumber
\end{array}
\ee
where  $u=(u_1,\ldots,u_n)$ and $f = (f_1,\dots,f_n)$ are  vectors of real-valued functions,
$A=\diag(A_1,\dots,A_n)$ and $B = \{B_{jk}\}_{j,k=1}^n$ are matrices of real-valued functions,
$0\le m\le n$ are fixed integers, and
$p_{jk}$  are real constants.

The purpose of the paper is to establish conditions 
ensuring existence and
 uniqueness
of small  global   classical (continuously differentiable) solutions   to the problem  \reff{eq:1}--\reff{eq:3} in the strip
$$
\Pi = \{(x,t)\in\R^2\,:\,0<x<1\}.
$$
If the coefficients of the hyperbolic system are almost periodic (or periodic) in $t$, we prove that the 
bounded solution is
almost periodic (respectively, periodic) also.

Denote by
 $\|\cdot\|$ the Euclidian norm in $\R^n$.
Given a (closed) domain $\Omega\subset\R^l$,
let $BC(\Omega;\R^n)$
be the Banach space of
all bounded and continuous maps
$u:\Omega \to \R^n$
with the usual $\sup$-norm
$$
\|u\|_{BC(\Omega;\R^n)}=\sup\left\{\|u(z)\|\,:\,z\in\Omega\right\}.
$$
Similarly one can  introduce the space $BC^k(\Omega;\R^n), k =1,2$, of bounded and $k$-times 
continuously differentiable functions.
 \vspace{3mm}

Suppose that the coefficients of the system
\reff{eq:1} satisfy the following conditions.
\vskip2mm
{\bf(H1)} \ There exists $\de_0 > 0$  such that
\begin{itemize}
\item
for all  $j\le n$ and  $k\le n$ the coefficients $A_j(x,t,v)$ and $B_{jk}(x,t,v)$ have bounded and
continuous partial derivatives  up to the second order in $(x,t)\in\overline\Pi$ and in $v\in\R^n$
with $\|v\| \le \de_0$,
\item
there exists $\Lambda_0>0$ such that
\begin{equation}\label{eq:h1}
\begin{array}{ll}
\inf\left\{A_j(x,t,v)\,:\,(x,t)\in[0,1]\times\R, \|v\| \le \de_0, 1\le j\le m\right\}\ge \Lambda_0,\\ [2mm]
\sup\left\{A_j(x,t,v)\,:\,(x,t)\in[0,1]\times\R, \|v\| \le \de_0, m+1\le j\le n\right\}\le -\Lambda_0,\\ [2mm]
\inf\left\{|A_j(x,t,v)-A_k(x,t,v)|\,:\,(x,t)\in[0,1]\times\R, \|v\| \le \de_0, 1\le j\ne k\le n\right\}\ge \Lambda_0.
\end{array}
\end{equation}
\end{itemize}
\vskip-3mm

{\bf (H2)} \  For all  $j\le n$ the functions $f_j(x,t)$   have bounded and
continuous partial derivatives  up to the second order in  $(x,t)\in\overline\Pi$.
\vskip3mm
Along with the nonlinear system  (\ref{eq:1}), consider its linearized version at   $u=0$, namely
\begin{equation}\label{eq:1u}
\partial_tu  + a(x,t)\partial_x u + b(x,t) u = 0, \quad x\in(0,1),
\end{equation}
where $a(x,t)=A(x,t,0)$ and $b(x,t)=B(x,t,0)$. Supplement the system \reff{eq:1u}
with the boundary conditions (\ref{eq:3}) and the initial conditions
\begin{eqnarray}
u(x,s) = \varphi(x), \quad x\in[0,1],
\label{eq:in}
\end{eqnarray}
where $s\in\R$ is an arbitrary fixed initial time.

We will work with the evolution family generated by the problem (\ref{eq:1u}), (\ref{eq:3}), (\ref{eq:in})
and defined on $L^2\left((0,1);\R^n\right)$. To introduce the evolution family,
let us define the notion of  an
$L^2$-generalized  solution.

Let $
C_0^1([0,1];\R^n)
$ be the space of continuously differentiable functions on $[0,1]$ with compact support in $(0,1)$. It is evident that the functions from $
C_0^1([0,1];\R^n)
$ fulfill the zero-order and the first-order compatibility conditions between  \reff{eq:in} and \reff{eq:3}.
Hence,
due to  Theorem~\ref{km} in Section \ref{sec:auxiliarli}, if $\vphi\in C_0^1([0,1];\R^n)$,  then the problem  (\ref{eq:1u}), (\ref{eq:3}), (\ref{eq:in})
has a unique classical solution.

 \begin{defn}\label{L2}
Let $\vphi\in L^2((0,1);\R^n)$. A function $u\in C\left([s,\infty), L^2\left((0,1);\R^n\right)\right)$ is called
an {\rm $L^2$-generalized  solution} to the problem (\ref{eq:1u}), (\ref{eq:3}), (\ref{eq:in})
 if for any sequence 
$\vphi^l\in C_0^1([0,1];\R^n)$ with 
 $\vphi^l\to\vphi$ in $L^2\left((0,1);\R^n\right)$
 the sequence $u^l\in C^1\left([0,1]\times[s,\infty);\R^n\right)$
of  classical solutions to
(\ref{eq:1u}), (\ref{eq:3}), (\ref{eq:in}) with
 $\vphi$ replaced by $\vphi^l$   fulfills the convergence
$$
\|u(\cdot,\theta)-u^l(\cdot,\theta)\|_{L^2\left((0,1);\R^n\right)} \to 0 \mbox{ as } l\to\infty,
$$
uniformly in $\theta$ varying in the range $s\le\theta\le s+T$, for every $T>0$.
\end{defn}

As usual, by $\LL(X,Y)$ we denote the space of linear
bounded operators from $X$ into $Y$, and write $\LL(X)$ for $\LL(X,X)$.
Note that the assumption {\bf (H1)} (especially, \reff{eq:h1}) entails
that
\begin{equation}\label{eq:h11}
\begin{array}{ll}
\inf\left\{a_j(x,t)\,:\,(x,t)\in[0,1]\times\R,  1\le j\le m\right\}\ge \Lambda_0,\\ [2mm]
\sup\left\{a_j(x,t)\,:\,(x,t)\in[0,1]\times\R,  m+1\le j\le n\right\}\le -\Lambda_0,\\ [2mm]
\inf\left\{|a_j(x,t)-a_k(x,t)|\,:\,(x,t)\in[0,1]\times\R, 1\le j\ne k\le n\right\}\ge \Lambda_0.
\end{array}
\end{equation}

\begin{thm}\label{evol}\cite{KmLyul}
Suppose that the coefficients  $a$  and $b$ of the system (\ref{eq:1u})
 have bounded and
continuous partial derivatives  up to the first order in $(x,t)\in\overline\Pi$. If
the inequalities \reff{eq:h11} are fulfilled, then, given $s\in\R$ and $\vphi\in L^2\left((0,1);\R^n\right)$, there exists a unique
$L^2$-generalized  solution $u:\R^2\to\R^n$ to the problem
(\ref{eq:1u}), (\ref{eq:3}), (\ref{eq:in}). Moreover, the map
$$
\vphi\mapsto U(t,s)\vphi:=u(\cdot,t)
$$
from $L^2\left((0,1);\R^n\right)$ to itself
defines  a strongly continuous, exponentially bounded evolution family  $U(t,s)\in \LL\left(L^2\left((0,1);\R^n\right)\right)$, which means that
\begin{itemize}
\item
$U(t,t)=I$ and $U(t,s)=U(t,r)U(r,s)$  for all $t\ge r\ge s,$
\item
 the map $(t,s)\in \R^2\mapsto U(t,s)\vphi\in L^2\left((0,1);\R^n\right)$ is continuous for all $t\ge s$ and each
$\vphi\in L^2\left((0,1);\R^n\right)$,
\item
  there exist $K \ge 1$ and $\nu \in \R$ such that
\beq\label{k3}
\|U(t,s)\|_{\LL\left(L^2\left((0,1);\R^n\right)\right)}\le Ke^{\nu(t-s)} \mbox{ for all } t\ge s.
\ee
\end{itemize}
\end{thm}

We will consider boundary conditions  ensuring
that the regularity  of   solutions to the initial boundary value problem for the linearized system
increases in a finite time. In other words,
 we  assume that
the system  (\ref{eq:1u}), (\ref{eq:3}), (\ref{eq:in})   has a smoothing property of the following kind, see
 \cite{kmit,Km,KmLyul}.

\begin{defn}\label{defn:smoothing_abstr}
Let $Y\hookrightarrow Z$ be continuously
embedded Banach spaces and, for each $s$ and $t\ge s$,  $V(t,s)\in\LL(Z)$.
The two-parameter family
$\{V(t,s)\}_{t\ge s}$
is called {\rm  smoothing} from $Z$ to $Y$ if there is $T>0$ (smoothing time) such that
 $V(t,s)\in \LL(Z,Y)$ for all $t\ge s+T$.
\end{defn}

Now we introduce the following condition ensuring a smoothing property 
of the evolution family
generated by the problem (\ref{eq:1u}), (\ref{eq:3}), (\ref{eq:in}) (see Theorem \ref{lem:d} below).
\vskip3mm

${\bf (H3)}$\
 $p_{i_1i_2}p_{i_2i_3}\dots p_{i_{n}i_{n+1}}=0$ for all tuples
$\left(i_1,i_2,\dots,i_{n+1}\right)\in\left\{1,\dots,n\right\}^{n+1}$.

\begin{defn} \label{defn:dich}\cite{Latn,H}
An evolution family $\{U(t,s)\}_{t\ge s}$ on a Banach space $X$ is said to have an exponential dichotomy on
 $\R$  (with an exponent $\al > 0$ and a bound $M\ge 1$) if there exists a
projection-valued function $P: \R\to\LL(X)$
such that the function $t\mapsto P(t)x$ is continuous and bounded for each $x\in X$,
and for all $t \ge s$ the following hold:

(i) $U(t,s)P(s) = P(t)U(t,s)$;

(ii) $U(t,s)(I-P(s))$ is invertible as an operator from
 $\Im(I-P(s))$ to $\Im(I-P(t))$ with the inverse denoted by $U(s,t)$;

(iii) $\|U(t,s)P(s)\|_{\LL(X)} \le M e^{-\al(t-s)}$;

(iv) $\| U(s,t)(I-P(t))\|_{\LL(X)} \le M e^{-\al(t-s)}$.
\end{defn}

We are prepared to state the main result of the paper.

\begin{thm}\label{main}
Suppose that the assumptions $\bf(H1)$--$\bf(H3)$ are fulfilled. Moreover, suppose that
the
evolution family $\{U(t,s)\}_{t\ge s}$ in $L^2\left((0,1);\R^n\right)$ generated by the
linearized problem
(\ref{eq:1u}), (\ref{eq:3}) has an exponential dichotomy   on
 $\R$.
 Then the following is true:

$(\io)$ 
There exist $\eps>0$ and $\delta>0$ such that for all $f\in BC^2\left(\overline\Pi;\R^n\right)$
with $\|f\|_{BC^2\left(\Pi;\R^n\right)}\le\eps$  there exists a unique classical solution $u$ to the problem
(\ref{eq:1}), (\ref{eq:3}) such that $\|u\|_{BC^1\left(\Pi;\R^n\right)}\le\delta$.

$(\io\io)$ If  the coefficients $A(x,t,v)$,
$B(x,t,v)$, and $f(x,t)$ are Bohr almost periodic in $t$ uniformly in $x\in[0,1]$
and $v$ with $\|v\|\le\de_0$
(respectively, $T$-periodic in $t$),
then the bounded classical solution to
the  problem  (\ref{eq:1}), (\ref{eq:3}) is Bohr almost periodic in $t$ (respectively, $T$-periodic in $t$)
 as well.
\end{thm}

The paper is organized as follows.
In Section \ref{sec:ex} we describe our approach and discuss the assumptions 
of our main Theorem \ref{main}.
In Section \ref{sec:robust} we obtain a general result about  robustness of the exponential dichotomy 
for the
linearized problem   under small perturbations of all coefficients in the 
hyperbolic system.
In Section~\ref{sec:abstr} we prove the equivalence of the PDE  setting 
(\ref{eq:1}), (\ref{eq:3}), 
(\ref{eq:in}) and the corresponding abstract setting.
The main result of the paper, stated in Theorem \ref{main}, is proved in Section \ref{sec:main}.

\section{Motivation and comments}\label{sec:ex}
\renewcommand{\theequation}{{\thesection}.\arabic{equation}}
\setcounter{equation}{0}

An overview of known existence-uniqueness results on global regular  solutions for first-order
one-dimensional hyperbolic systems of quasilinear equations can be found e.g. in~\cite{li,MyFi}.

\subsection{Our setting}

\subsubsection{Quasilinear system (\ref{eq:1})}
Quasilinear hyperbolic system (\ref{eq:1}) is written in  the  canonical form
of Riemann invariants. Specifically, 
 the matrix $A$ is diagonal and each equation of the system consists of partial derivatives of 
a single unknown function only. 
R.~Courant and P.~Lax \cite{CL}  showed that many quasilinear
one-dimensional hyperbolic  systems can be written in the form of (\ref{eq:1}). 
Moreover, B. Rozhdestvenskii and N. Yanenko \cite{RoJa} 
showed that even more general nonlinear systems of the kind
$\d_tu=g(x,t,u,\d_xu)$ are reducible to~(\ref{eq:1}). 

 \subsubsection{Smoothing boundary conditions}
 Assumption $\bf(H3)$ on the boundary conditions (\ref{eq:3}) is essentially 
 used in the proofs of the robustness Theorem \ref{robust}
 and the main Theorem \ref{main}.
  It turns out that
$\bf(H3)$ has an algebraic characterization being equivalent to 
\beq \label{w*}
\tr\left(\sum\limits_{k=1}^nW^k\right)=0,
\ee
where $W$ is the $n\times n$-matrix with entries $w_{jk}=|p_{jk}|$. This characterization implies that the assumption $\bf(H3)$ is efficiently checkable. The proof of the
equivalence of $\bf(H3)$ and \reff{w*} and a comprehensive
discussion of boundary conditions of this type can be found in \cite{kl}.

Problems (\ref{eq:1}), (\ref{eq:3}) satisfying Assumption $\bf(H3)$ occur in
chemical kinetics \cite{Zel,Zel2}, population dynamics \cite{lakra,pavel}, boundary feedback control theory
\cite{BC,CBdA,gugat1,pavel}, and inverse problems \cite{TiNg}. A collection of examples from these areas can be found in \cite{KmLyul}.

\subsection{Robustness of exponential dichotomy}
Since the nonlinear coefficients  $A$ and $B$ are a source of different singularities,
global classical solvability requires 
assumptions preventing shocks and blow-ups.
To construct small global regular  solutions to the quasilinear system \reff{eq:1}, \reff{eq:3}, we assume 
the smallness of the right-hand sides and a regular behavior of the linearized system. The latter is 
ensured by the existence
of the exponential dichotomy on $\R$
for the evolution family on $L^2((0,1);\R^n).$ A crucial technical tool in our analysis  is
the robustness of exponential dichotomy for perturbations of $a$ and $b$.
Though robustness issue has extensively been studied in the literature \cite{ChL,ChL2,H,Pal,PS}, 
 none of these results is
applicable to hyperbolic PDEs with unbounded perturbations.
D.~Henry  established a general sufficient condition 
of the robustness for abstract evolution equations
(see Theorem \ref{robust_gen} below). Attempts to apply this approach
to hyperbolic PDEs
meet complications caused by loss of regularity.
It turns out that the loss of regularity is unavoidable for perturbations of 
the coefficients $a_j$ (unbounded perturbations).
 In \cite{KRT} these complications are overcome 
for the boundary conditions of the 
smoothing type in the space of continuous functions. Here we extend our approach to the 
 $L^2$-setting. In \cite{kkr} this issue is addressed for the 
 periodic boundary conditions.
 For more general boundary conditions  the robustness issue
for hyperbolic PDEs remains unexplored.

\subsection{Our approach and the choice of spaces}
In the proof of our main Theorem \ref{main} we use an iteration procedure to construct  
classical (continuously differentiable) solutions.
Each iteration is a $C^2$-solution to  the corresponding linear problem with coefficients depending on the 
preceding iteration. To solve this linear problem,  we put it into an abstract $L^2$-setting, which is provable to be equivalent to the $L^2$-setting
in the sense of Definition \ref{L2}. Using an $L^2$-setting 
instead of the smooth setting enables us to use appropriate results from the abstract theory of 
evolution semigroups.
Due to the robustness Theorem \ref{robust}, the homogeneous 
version of the linear problem under consideration has an exponential 
dichotomy on $\R$. Consequently, the nonhomogeneous problem 
admits a unique solution given by  Green's formula whenever the  right-hand side belongs to the domain of the corresponding evolution family. 
This means that, working in the spaces of continuous functions, the  right-hand sides  have to satisfy 
compatibility conditions
for all $t\in\R$ (what cannot be fulfilled on each step of our iteration procedure), while working in $L^2$ the compatibility
conditions are not needed at all. Finally, we show that the 
$L^2$-solution is actually in $C^2$, for which we use the
smoothing property provided by Theorem \ref{lem:d}.

\subsection{Time-periodic problems}
A  natural way of proving   the existence of time-periodic solutions is provided by  
local smooth continuation theory and
bifurcation theory. In \cite{KmRe4} a generalized implicit function theorem is established 
to prove the existence of time-periodic solutions and in
 \cite{KmRe3} the  Lyapunov-Schmidt reduction is adapted to prove the existence of Hopf 
bifurcations for  {\it semilinear} hyperbolic problems.
We suggest  another approach and provide a constructive method of getting periodic solutions,
 for {\it quasilinear} hyperbolic PDEs. 
Existence  of time-periodic solutions to
nonlinear hyperbolic PDEs is a challenging problem going back to 
the classical work by E. Fermi, et al. \cite{FPU}.
Verificating the hypothesis of P. Dedye \cite{deb} numerically, they observed the 
existence of time-periodic solutions in nonlinear
hyperbolic problems.

Analysis of time-periodic solutions to hyperbolic PDEs
usually  meets a complication known as a problem of small divisors.
However,  if  boundary conditions are of smoothing type, 
then this problem does not appear at all. For the discussion of this point see \cite{kl}.

\subsection{Verification of the assumption that the linearized problem
has an exponential dichotomy on $\R$}
While it is relatively easy to verify the assumptions $\bf(H1)$, $\bf(H2)$, and $\bf(H3)$ of
our main Theorem \ref{main},
it is not so trivial to verify  the  remaining assumption that the evolution family $U(t,s)$
has an exponential dichotomy on $\R$, see e.g. \cite{kkr}. In particular, if the coefficients $a$ and $b$ do not
depend on $t$, then for the problem  \reff{eq:1u}, \reff{eq:3}
(where the boundary conditions are considered to be of a smoothing type) falls into the scope
of the spectral mapping theorem for eventually differentiable $C_0$-semigroups.
This means that  the exponential dichotomy is 
described by spectral properties
of the corresponding operator, what is described in detail 
in the next example.

\begin{ex} \label{contr} 
\rm
Our aim is to show that the  assumption about the existence of an 
exponential dichotomy
does not
contradict to $\bf(H1)$, $\bf(H2)$, and $\bf(H3)$. Furthermore, we will show that the dichotomy is not necessary trivial
or,
in other words,
the dichotomous system  \reff{eq:1u}, \reff{eq:3}
is not necessarily exponentially stable.

Consider the following $2\times 2$-first-order quasilinear hyperbolic system depending on a real parameter $\lambda$:
\begin{equation} \label{eq:1ex}
\begin{array}{ll}
 \partial_t u_1 + A_1(x,t,u)\partial_x u_1  = (\lambda + B_{11}(x,t,u))u_1 + (B_{12}(x,t,u) - 1)u_2 + f_1(x,t) \\ [1mm]
 \partial_t u_2 + A_2(x,t,u)\partial_x u_2  = B_{21}(x,t,u)u_1 + B_{22}(x,t,u)u_2 + f_2(x,t)
\end{array}
\end{equation}
with the boundary conditions
\begin{equation}  \label{eq:3ex}
u_1(0,t)=0, \quad u_1(1,t)=u_2(1,t),
\end{equation}
where
$A_1(x,t,0) \equiv 1, \ A_2(x,t,0) \equiv -1,$ and  $B_{ij}(x,t,0) \equiv 0$ with $ i,j = 1,2$.
Assume that the functions $A_j$, $B_{jk}$, and $f_j$ fulfill the conditions $\bf (H1)$ and $\bf (H2)$, what causes that
the coefficients of \reff{eq:1ex} fulfill $\bf (H1)$ and $\bf (H2)$ as well.
It is evident  that the  boundary conditions \reff{eq:3ex} fulfill the  condition $\bf (H3)$.

In order to check the remaining assumption, let us consider the  linearized problem
\begin{eqnarray} \label{lin1}
 & & \partial_t u_1 + \partial_x u_1  = \lambda u_1 - u_2, \  \partial_t u_2 - \partial_x u_2  = 0, \\
 & & u_1(0,t)=0 , \quad u_1(1,t)=u_2(1,t). \label{lin2}
\end{eqnarray}
Our aim is to state conditions on $\la$ under which the system \reff{lin1}--\reff{lin2} is dichotomous on~$\R$.
The corresponding eigenvalue problem reads
\beq
\label{exampleevp}
\begin{array}{l}
v_1'-\la v_1 + v_2 = \mu v_1, \ v_2'= \mu v_2, \ x \in (0,1),\\
v_1(0)= 0, \ v_1(1)-v_2(1)=0,
\end{array}
\ee
 $\mu$ being the spectral parameter.
It is easy to verify that there do not exist real eigenvalues  to \reff{exampleevp} and that \reff{exampleevp} is equivalent to
\begin{eqnarray*}
 & & v_1(x)=\frac{c}{\la-2\mu}\left(e^{\mu x} -e^{(\la-\mu)x}\right),\; v_2(x)=ce^{\mu x}, \\
& & e^{\la-2\mu}=2\mu-\la+1.
\end{eqnarray*}
Here $c=v_2(0)$ is a nonzero complex constant.
Setting $\la-2\mu=\xi+i\eta$ with $\xi \in \R$ and (without loss of generality) $\eta>0$, we get
$$
\eta=\sqrt{e^{2\xi}-(1-\xi)^2}
$$
and
\beq
\label{chareq}
\sin \sqrt{e^{2\xi}-(1-\xi)^2}=- \sqrt{1-e^{-2\xi}(1-\xi)^2}.
\ee
It is easy to see that equation \reff{chareq} has (besides of the solution $\xi=0$) a countable number of
solutions $0<\xi_0<\xi_1\ldots$ tending to $\infty$.
Hence, the spectrum
of \reff{exampleevp} consists of countably many geometrically simple eigenvalues
$$
\mu_j^\pm(\la)=\frac{1}{2}\left(\la-\xi_j\pm i\sqrt{e^{2\xi_j}-(1-\xi_j)^2}\right).
$$
If $\lambda \not= \xi_j$ for all $j\in\N$, then the real parts of all eigenvalues
are not equal to zero. By
Theorem \ref{lem:d}, the evolution semigroup on $L^2\left((0,1);\R^n\right)$ generated by the
linearized problem (\ref{lin1}), (\ref{lin2})
 is eventually differentiable and, hence by  \cite[p. 281, Corollary 3.12]{EN}, satisfies the spectral mapping theorem.
This entails that the system (\ref{lin1}), (\ref{lin2})
is exponentially dichotomous on $\R$ with an exponent $\alpha=\alpha(\lambda)$  fulfilling the inequality
 $\alpha(\lambda) \le \min_j|\lambda - \xi_j|.$
Furthermore, if $\xi_k < \lambda < \xi_{k+1}$, then the system \reff{eq:1ex}--\reff{eq:3ex} has a $k$-dimensional unstable submanifold.
\end{ex}

\section{Robustness of  exponential dichotomy}
\label{sec:robust}

\renewcommand{\theequation}{{\thesection}.\arabic{equation}}
\setcounter{equation}{0}

\subsection{Auxiliary statements}\label{sec:auxiliarli}

We start with providing
 existence-uniqueness results for the
homogeneous system (\ref{eq:1u}) and
its non-homogeneous version
\begin{equation}\label{eq:vg}
\partial_tu  + a(x,t)\partial_x u + b(x,t) u = f(x,t), \quad x\in(0,1),
\end{equation}
both subjected to the boundary conditions (\ref{eq:3})
and the initial conditions
 (\ref{eq:in}).

Given $s\in\R$, denote
$$
\Pi_{s} = \{(x,t)\,:\,0<x<1, s<t<\infty\}.
$$
The existence and uniqueness of  classical and piecewise smooth solutions
 to  initial-boundary value hyperbolic problems
is proved in
\cite{ijdsde}. We summarize the needed results in the following theorem.

\begin{thm}\label{km}
Suppose that the coefficients  $a$  and $b$ of the system (\ref{eq:1u})
are continuous and   have bounded and
continuous first-order partial derivatives  in $x$. Moreover,
suppose that the condition \reff{eq:h11}  is fulfilled.
Let $s \in \R$ be  arbitrarily fixed and $\varphi\in C^1([0,1];\R^n)$.

$(\io)$
If  $f$ is continuous and   has bounded and
continuous first-order partial derivatives  in $x$, and $\varphi$ fulfills the zero order compatibility conditions
\beq\label{eq:nl1}
\begin{array}{l}
\displaystyle
\varphi_j(0) =
\sum\limits_{k=1}^mp_{jk}\varphi_k(1)+\sum\limits_{k=m+1}^np_{jk}\varphi_k(0),
\quad  1\le j\le m,
\nonumber\\
\displaystyle
\varphi_j(1) = \sum\limits_{k=1}^mp_{jk}\varphi_k(1)+\sum\limits_{k=m+1}^np_{jk}\varphi_k(0),
\quad   m< j\le n,
\nonumber
\end{array}
\ee
 then in $\Pi_s$ there exists a unique continuous  solution $u(x,t)$ to the problem
\reff{eq:vg}, (\ref{eq:3}), (\ref{eq:in}) that is a piecewise continuously differentiable function (further referred to as {\rm piecewise continuously differentiable
solution}).

$(\io\io)$
If  $\varphi$ fulfills the zero order compatibility conditions \reff{eq:nl1}
and the
first order compatibility conditions
\begin{eqnarray*}
& & \psi_j(0) =
\sum\limits_{k=1}^mp_{jk}\psi_k(1)+\sum\limits_{k=m+1}^np_{jk}\psi_k(0),
\quad  1\le j\le m,  \\
& & \psi_j(1) = \sum\limits_{k=1}^mp_{jk}\psi_k(1)+ \sum\limits_{k=m+1}^np_{jk}\psi_k(0),
\quad   m< j\le n,
\end{eqnarray*}
where
$$
\psi(x)= -(a(x,s)\partial_x + b(x,s))\varphi(x),
$$
then in $\Pi_s$ there exists a unique classical  solution $u(x,t)$ to the problem
(\ref{eq:1u}), (\ref{eq:3}), (\ref{eq:in}). Moreover,  there are
 constants $K_1\ge 1$ and $\nu_1>0$ not depending on $s$, $t$, and $\varphi$ such that
 $$
\|u(\cdot,t)\|_{C^1([0,1];\R^n)}\le K_1e^{\nu_1(t-s)}\|\vphi\|_{C^1([0,1];\R^n)}\,\,\mbox{ for }\,\,
 t\ge s.
$$
\end{thm}

Similarly to the homogeneous problem (\ref{eq:1u}), (\ref{eq:3}),
(\ref{eq:in}),  we introduce the notion of an $L^2$-generalized  solution
for the non-homogeneous problem \reff{eq:vg},
  (\ref{eq:3}), (\ref{eq:in}).

\begin{defn}\label{L2_g}
Let $\vphi\in L^2((0,1);\R^n)$.  A function $u\in C\left([s,\infty), L^2\left((0,1);\R^n\right)\right)$ is called
an {\rm $L^2$-generalized  solution} to the problem (\ref{eq:vg}), (\ref{eq:3}), (\ref{eq:in})
if for any  sequence $\vphi^l\in C_0^1([0,1];\R^n)$ with $\vphi^l\to\vphi$ in $L^2\left((0,1);\R^n\right)$
the  sequence $u^l$
of   piecewise continuously differentiable solutions to 
(\ref{eq:vg}), (\ref{eq:3}), (\ref{eq:in}) with
 $\vphi(x)$ replaced by $\vphi^l(x)$   fulfills the convergence
$$
\|u(\cdot,\theta)-u^l(\cdot,\theta)\|_{L^2\left((0,1);\R^n\right)} \to 0 \mbox{ as } l\to\infty,
$$
uniformly in $\theta$ varying in the range $s\le\theta\le s+T$, for every $T>0$.
\end{defn}

We will use the following variant of the existence-uniqueness result stated in
\cite[Theorem 2.3
]{KmLyul},  for the case of the non-homogeneous system \reff{eq:vg}.

\begin{thm}\label{evol_g}
Suppose that the coefficients  $a$, $b$, and $f$ of the system \reff{eq:vg}
 have bounded and
continuous partial derivatives  up to the first order in $(x,t)\in\overline\Pi$.
Moreover,
suppose that the condition \reff{eq:h11}  is fulfilled.
Then, given $s\in\R$ and $\vphi\in L^2\left((0,1);\R^n\right)$, there exists a unique
$L^2$-generalized  solution $u:\R^2\to\R^n$ to the problem
(\ref{eq:vg}), (\ref{eq:3}), (\ref{eq:in}).
\end{thm}

The proof of this theorem repeats the proof of  \cite[Theorem 2.3]{KmLyul}.

As it follows from the results of \cite{kmit,Km,KmLyul},
the problems  (\ref{eq:1u}), (\ref{eq:3}), (\ref{eq:in}) and (\ref{eq:vg}), (\ref{eq:in}), (\ref{eq:3}) have
 a smoothing property,  described in the next two theorems.

\begin{thm}\label{lem:d} 
Let the assumption {\bf (H3)} and the conditions of Theorem \ref{evol} be fulfilled.
Then there exists $d>0$ not depending on $s\in\R$ such that

$(\io)$
the evolution family
$\{U(t,\tau)\}_{t\ge \tau}$ on  $L^2\left((0,1);\R^n\right)$ generated by (\ref{eq:1u}), (\ref{eq:3})
is smoothing from $L^2\left((0,1);\R^n\right)$ to  $C^1([0,1];\R^n)$, with  smoothing time equal to $2d$.

$(\io\io)$
if $a$ and $b$ have bounded and continuous partial derivatives up to the second order in $(x,t)\in\overline\Pi$,
then the evolution family
$\{U(t,\tau)\}_{t\ge \tau}$ on  $L^2\left((0,1);\R^n\right)$ generated by (\ref{eq:1u}), (\ref{eq:3})
is smoothing from $L^2\left((0,1);\R^n\right)$ to  $C^2([0,1];\R^n)$, with smoothing time  equal to $3d$.
\end{thm}

\begin{thm}\label{lem:2d}
Let the assumption {\bf (H3)} and the conditions of Theorem \ref{evol} be fulfilled.
Let $s\in\R$ be arbitrary fixed.
If   $f\in BC^1(\overline\Pi;\R^n)$, then the $L^2$-generalized solution $u(x,t)$ to the problem
(\ref{eq:vg}), (\ref{eq:3}), (\ref{eq:in})
 is $C^1$-smooth after  time $s+2d$ and  satisfies the estimate
\begin{eqnarray}
\|u(\cdot,t)\|_{C^1\left([0,1];\R^n\right)}
\le 
  L \left(\|\vphi\|_{L^2((0,1);\R^n)} + \|f\|_{BC^1(\Pi;\R^n)} \right),\quad
  s+2d\le t\le s+3d.\label{apr_vv1}
\end{eqnarray}
If $f\in BC^2(\overline\Pi;\R^n)$, then  $u(x,t)$ is $C^2$-smooth after  time $s+3d$ and  satisfies the estimate
\begin{eqnarray}
&&\|u(\cdot,s+3d)\|_{C^2\left([0,1];\R^n\right)}
\le L \left(\|\vphi\|_{L^2((0,1);\R^n)} + \|f\|_{BC^2(\Pi;\R^n)} \right). \label{apr_vv}
\end{eqnarray}
Here
the constant $L>0$ depends on $d$ but does not depend on the initial time $s\in\R$, the
initial function $\varphi \in L^2((0,1);\R^n)$, and the coefficient $f$.
\end{thm}

One of our  main technical tools is the robustness of an exponential dichotomy on $\R$
(Theorem~\ref{robust} below). To prove this result, we will check the following modification
of the sufficient condition
 established by D. Henry in \cite[Theorem 7.6.10]{H}, see 
 \cite[Theorem 2.3]{KRT}.

\begin{thm}\label{robust_gen}
Let $X$ be a Banach space.
Assume that the evolution operator $U(t,s)\in\LL\left(X\right)$ has an exponential dichotomy on
 $\R$ with an exponent $\alpha$ and a bound $M$. Assume also that 
$\|U(t,s)\|_{\LL(X)}$ is bounded by a constant over all $s,t$ such that $0\le t-s\le 1$.
Then there exist positive $\eta$, $T$,  $\alpha_1 \le \alpha$, and
$M_1 \ge M$ such that every perturbed evolution operator $\tilde U(t,s)\in\LL\left(X\right)$ with
$$
\|U(t,s) - \tilde U(t,s)\|_{\LL\left(X\right)} < \eta, \ \ {\rm whenever} \ \ t-s = T 
$$
has an exponential dichotomy on  $\R$ with an exponent  $\alpha_1 $ and a bound
$M_1$.
\end{thm}

In the proof of the robustness Theorem~\ref{robust}, by technical reasons instead of the 
constructive condition ${\bf(H3)}$ we will use a non-constructive  condition stated below
as  ${\bf(H3)^\prime}$.  Our nearest goal is to introduce ${\bf(H3)^\prime}$ and to
 show that ${\bf(H3)}$ entails ${\bf(H3)^\prime}$. 

To this end, let us introduce a weak formulation of the problem (\ref{eq:1u}), (\ref{eq:3}), (\ref{eq:in})
 using integration along characteristic curves.
 For given $j\le n$, $x \in [0,1]$, and $t \in \R$, the $j$-th characteristic of \reff{eq:1u}
passing through the point $(x,t)\in\overline\Pi_s$ is defined
as the solution 
$$
\xi\in [0,1] \mapsto \om_j(\xi)=\om_j(\xi,x,t)\in \R
$$ 
of the initial value problem
\beq\label{char}
\partial_\xi\om_j(\xi, x,t)=\frac{1}{a_j(\xi,\om_j(\xi,x,t))},\;\;
\om_j(x,x,t)=t.
\ee
Due to the assumption \reff{eq:h11}, the characteristic curve $\tau=\om_j(\xi,x,t)$ reaches the
boundary of $\Pi_s$ in two points with distinct ordinates. Let $x_j(x,t)$
denote the abscissa of that point whose ordinate is smaller.
Remark that  the value of  $x_j(x,t)$
does not depend on $x,t$ if $t>s+\frac{1}{\Lambda_0}$. 
More precisely, it holds
\beq\label{*k}
x_j(x,t)=x_j=\left\{
 \begin{array}{rl}
 0 &\mbox{if}\ 1\le j\le m\\
 1 &\mbox{if}\ m<j\le n
\end{array}
\right.
\quad\mbox{for } t>\frac{1}{\Lambda_0}.
\ee

Write
\begin{eqnarray*}
c_j(\xi,x,t)=\exp \int_x^\xi
\left[\frac{b_{jj}}{a_{j}}\right](\eta,\om_j(\eta))\,d\eta,\quad
d_j(\xi,x,t)=\frac{c_j(\xi,x,t)}{a_j(\xi,\om_j(\xi))}.
\end{eqnarray*}
Introduce a linear bounded operator
$R: BC\left(\overline \Pi;\R^n\right)\mapsto BC\left(\R;\R^n\right)$ by
\begin{eqnarray}
\begin{array}{ll}
\displaystyle (Ru)_j(t)=\sum\limits_{k=m+1}^np_{jk} u_k(0,t)+\sum\limits_{k=1}^mp_{jk} u_k(1,t),\quad  j\le n,
\end{array}\label{eq:R}
\end{eqnarray}
and
an affine bounded operator
$Q: BC\left(\overline\Pi_s;\R^n\right)\mapsto BC\left(\overline\Pi_s;\R^n\right)$ by
\begin{eqnarray}
\label{Q}
(Qu)_j(x,t)=
\left\{\begin{array}{lcl}
c_j(x_j(x,t),x,t)(Ru)_j(\om_j(x_j(x,t))) & \mbox{if}&  x_j(x,t)\notin(0,1), \\
c_j(x_j(x,t),x,t)\vphi_j(x_j(x,t))      & \mbox{if}&  x_j(x,t)\in(0,1),
\end{array}\right.
\end{eqnarray}
being defined on the affine subspace of $BC\left(\overline\Pi_s;\R^n\right)$ of functions satisfying the
initial condition (\ref{eq:in}).

A $C^1$-map
$u: \overline\Pi_s \to \R^n$ is a  classical solution to
(\ref{eq:vg}), (\ref{eq:3}), (\ref{eq:in})  if and only if
it satisfies the following system of integral equations
\beq
\label{rep}
u_j(x,t)=(Qu)_j(x,t)
-\int_{x_j(x,t)}^x d_j(\xi,x,t)
\Bigl[\sum\limits_{k\ne j}b_{jk}u_k - f_j\Bigl](\xi,\om_j(\xi)) d\xi, \quad  j\le n.
\ee
 A $C$-map
$u: \overline\Pi_s \to \R^n$ is called a {\it continuous} solution to  (\ref{eq:vg}), (\ref{eq:3}),
(\ref{eq:in})
in $\overline\Pi_s$ if it satisfies \reff{rep} in $\overline\Pi_s$.

Introduce  a linear
bounded operator
 $S$ from $BC(\R;\R^n)$ to $BC(\overline\Pi;\R^n)$ by
\beq\label{S}
 (Sv)_j(x,t)=c_j(x_j,x,t)v_j(\om_j(x_j,x,t)), \quad j\le n,
\ee
where $x_j$ is given by \reff{*k}.
A sufficient condition ensuring a smoothing property of the evolution family generated by
(\ref{eq:1u}), (\ref{eq:3}),  (\ref{eq:in}) (see Theorem \ref{lem:d}) can now be formulated as follows:
\vskip3mm
${\bf(H3)^\prime}$\
  $(SR)^nu\equiv 0$  for all $u\in C\left(\overline\Pi;\R^n\right)$.
\vskip3mm
This condition also means that  every  (continuous) solution to the decoupled system  (\ref{eq:1u})
($b_{jk}=0$ for all $k\ne j$) with the boundary and the initial conditions
(\ref{eq:3}) and (\ref{eq:in})
stabilizes to zero in a finite time.

\begin{lemma}\label{equiv}
 Condition 
$\bf(H3)^\prime$ follows from Condition   $\bf(H3)$.
\end{lemma}

\begin{proof}
First show that the lemma is true  for  $n=2$.
The condition  $\bf(H3)^\prime$ for  $n=2$ can be written as follows:
\begin{eqnarray*}
c_j(x_j,x,t) \left(RSRu\right)_j(\om_j(x_j,x,t))\equiv 0
\quad\mbox{for all }  (x,t)\in\overline\Pi,\, u\in BC\left(\overline\Pi;\R^2\right),\,
 j\le 2,
\end{eqnarray*}
that is equivalent  to
\begin{eqnarray}\label{*1}
 \left(RSRu\right)(t)\equiv 0
\quad\mbox{for all }   t\in\R, u\in BC\left(\overline\Pi;\R^2\right).
\end{eqnarray}
We have
\begin{eqnarray}
\lefteqn{
\displaystyle \left(RSRu\right)_j(t)=\sum\limits_{k=1}^2p_{jk}(SRu)_k(1-x_k,t)}
\nonumber\\
&&=\sum\limits_{k=1}^2p_{jk}c_k(x_k,1-x_k,t)(Ru)_k(\om_k(x_k,1-x_k,t))
\nonumber\\
&&=\sum\limits_{k=1}^2p_{jk}c_k(x_k,1-x_k,t)\sum\limits_{i=1}^2p_{ki}
u_i(1-x_i,\om_k(x_k,1-x_k,t)).\label{*2}
\end{eqnarray}

At the same time, the condition   $\bf(H3)$ for  $n=2$ reads
\beq\label{k6}
p_{jk}p_{ki}=0\quad \mbox{for all } \,  j,k,i\le 2.
\ee
As a consequence, the equations \reff{*2},  \reff{k6}, and  \reff{*1} entail the desired statement for $n=2$.

The proof for $n=3$ uses a similar argument. The analogs of \reff{*2} and \reff{k6} read
\begin{eqnarray}
& \displaystyle\left(RSRSRu\right)_j(t)=\sum\limits_{k=1}^3p_{jk}c_k(x_k,1-x_k,t)\sum\limits_{i=1}^3p_{ki}
c_i(x_i,1-x_i,\om_k(x_k,1-x_k,t))&\nonumber\\
&\displaystyle \times
\sum\limits_{s=1}^3p_{is}
u_s(1-x_s,\om_i(x_i,1-x_i,\om_k(x_k,1-x_k,t))&\label{*3}
\end{eqnarray}
and
\beq\label{k7}
p_{jk}p_{ki}p_{is}=0\quad \mbox{for all } \, j,k,i,s\le 3,
\ee
respectively. On the account of \reff{*3}, one can easily see that \reff{k7} implies $\bf(H3)^\prime$ for  $n=3$.

Proceeding similarly, one can easily obtain the desired statement for an arbitrary fixed $n\in\N$.
\end{proof}

\subsection{Robustness Theorem}

We here address the issue of robustness
 of the exponential dichotomy for the  linearized problem (\ref{eq:1u}), (\ref{eq:3}),
with respect to perturbations of the coefficients $a$ and $b$.
To this end, along with the system (\ref{eq:1u}) we will  consider its  perturbed version
\begin{equation}\label{eq:1p}
\partial_tv  + \left(a(x,t)+\tilde a(x,t)\right)\partial_x v + \bigl(b(x,t)+\tilde b(x,t)\bigr)
 v = 0, \quad x\in(0,1),
\end{equation}
where
$\tilde a=\diag(\tilde a_1,\dots,\tilde a_n)$ and
 $\tilde b=\{\tilde b_{jk}\}_{j,k=1}^n$ are matrices of real-valued functions.
Suppose that the entries of $\tilde a$ and $\tilde b$
 have bounded and
continuous partial derivatives  in $x$ and $t$ up to the second order.

Fix $\eps_0$ to be so small  that for all $\tilde a$ and $\tilde b$ with $\|\tilde a\|_{BC^1(\Pi;\M_n)}
\le \eps_0$ and
$\|\tilde b\|_{BC(\Pi;\M_n)} \le \eps_0$ the coefficients of the  system (\ref{eq:1p}) fulfill the assumptions of
Theorem \ref{evol} with $a$ and $b$ replaced by $a+\tilde a$ and $b+\tilde b$, respectively. This means that the
perturbed problem
\reff{eq:1p}, (\ref{eq:3}) generates the evolution family on $L^2\left((0,1);\R^n\right)$ (see Theorem \ref{evol}), 
which will be referred to as $\{\tilde U(t,s)\}_{t\ge s}$.
We also suppose  that the assumption   ${\bf (H3)}$ is fulfilled. Then Theorem~\ref{lem:d}$(\io)$ guarantees that the families
$\{U(t,s)\}_{t\ge s}$
and  $\{\tilde U(t,s)\}_{t\ge s}$ have a smoothing property in the sense of Definition \ref{defn:smoothing_abstr}.

\begin{thm}\label{robust}
Assume that
the evolution family $U(t,s)$  has an exponential dichotomy on~$\R$ with an exponent $\alpha>0$ and
a bound $M\ge 1$.
Then the value of  $\eps_0 > 0$ can be chosen so small that for all $\tilde a$ and $\tilde b$ with
$\|\tilde a\|_{BC^1(\Pi;\M_n)} \le \eps_0$ and
$\|\tilde b\|_{BC(\Pi;\M_n)} \le \eps_0$ the evolution family $\tilde U(t,s)$
 has an  exponential dichotomy on $\R$ with an exponent  $\alpha_1 \le \alpha$ and
a bound $M_1 \ge M$ depending on $\eps_0$
but not on $\tilde a$ and $\tilde b$.
\end{thm}

\begin{proof}
We check the sufficient conditions for the robustness of exponential dichotomies
given in Theorem \ref{robust_gen}.
Since the evolution family  $U(t,s)$ is exponentially bounded, the 
uniform boundedness of $\|U(t,s)\|_{\LL(X)}$  over all $s,t$ such that $0\le t-s\le 1$
follows directly from the estimate \reff{k3} and the assumption {\bf (H1)}.
It remains  to  prove that
there exists a function
$\beta:[0,1]\to\R$ with $\beta(\eps)\to 0$ as $\eps\to 0$,
 such that for all   $\tilde a$ and $\tilde b$ with
$\|\tilde a\|_{BC^1(\Pi;\M_n)}\le \eps$ and
$\|\tilde b\|_{BC(\Pi;\M_n)} \le \eps$
  we have
\beq\label{d1}
\bigl\|(U-\tilde U)(s+T,s)\bigr\|_{\LL\left(L^2\left((0,1);\R^n\right)\right)}\le\be(\eps)
\ee
for some $T>0$.

Recall that the operator $Q$ given by \reff{Q} is defined on the
affine subspace of $BC\left(\overline\Pi_s;\R^n\right)$ of functions satisfying the
initial condition (\ref{eq:in}).
It is important to note that $Q$ maps this subspace  into itself.
Due to {\bf (H1)}, {\bf (H3)}, and Lemma  \ref{equiv}, one can fix some $d>0$ such that
 $\left[Q^nu\right](x,s+d)=\left[(SR)^nu\right](x,s+d)\equiv 0$ for all
$x\in[0,1]$ and $s\in\R$.
Moreover, the value of $d$
remains the same, whenever the operators $S$, $Q$, and $R$ are perturbed by means of replacing $a$ and $b$
by $a+\tilde a$ and $b+\tilde b$ such that
$\|\tilde a\|_{BC^1(\Pi;\M_n)}\le \eps$ and
$\|\tilde b\|_{BC(\Pi;\M_n)} \le \eps$. On the account of  Theorem~\ref{lem:d} and Lemma \ref{equiv}, we conclude that
$U(s+d,s)\in\LL\left(L^2\left((0,1);\R^n\right),C\left([0,1];\R^n\right)\right)$
and
$U(s+2d,s)\in\LL\left(L^2\left((0,1);\R^n\right),C^1\left([0,1];\R^n\right)\right)$.

Fix an arbitrary  $s\in\R$.
Note that  any  initial function $\vphi\in C^1_0\left([0,1];\R^n\right)$, for which we have
$\vphi(0)=\vphi(1)=0$,
satisfies both the zero-order and the first-order compatibility conditions between
 (\ref{eq:3}) and \reff{eq:in}.
 Therefore, for given  $\vphi\in C^1_0\left([0,1];\R^n\right)$, by Theorem \ref{km}
there  exist  unique classical solutions $u$  and $v$
  to the problems (\ref{eq:1u}), (\ref{eq:3}),
\reff{eq:in}  and
(\ref{eq:1p}), (\ref{eq:3}), \reff{eq:in}, respectively.

Due to Theorem \ref{L2} and the fact that the space 
$C^1_0\left([0,1];\R^n\right)$ is dense in $L^2\left((0,1);\R^n\right)$, the desired estimate
 \reff{d1} will be proved if we derive the bound
\beq\label{main_estim}
\|(u-v)(\cdot,s+3d)\|_{L^2\left((0,1);\R^n\right)}\le \be(\eps)\|\vphi\|_{L^2\left((0,1);\R^n\right)}
\ee
uniformly in $s\in\R$,  $\vphi\in C^1_0\left([0,1];\R^n\right)$, and
$\tilde a,\tilde b\in BC^1\left( \overline\Pi;\M_n\right)$
with $\|\tilde a\|_{BC^1(\Pi;\M_n)}\le \eps$ and
$\|\tilde b\|_{BC(\Pi;\M_n)} \le \eps$. In \reff{main_estim} the number  $T$ is taken to be $3d$ by technical reasons.

We split the derivation of the estimate \reff{main_estim} into a sequence of steps.
\vskip2mm
{\it Step1.
Derivation of an equation for  $(u-v)\big|_{\overline\Pi_{s+3d}}$.
}
By the smoothing property,  after the time $t=s+2d$ the
solutions $u$ and $v$ are continuously differentiable and, therefore, satisfy pointwise the systems
(\ref{eq:1u}) and (\ref{eq:1p}), respectively . Our starting point is that  the difference $u-v$
fulfills the equation
$$
(\partial_t  + a(x,t)\partial_x + b(x,t))
 (u-v) = \tilde a(x,t)\partial_xv+ \tilde b(x,t)v,\quad (x,t)\in\Pi_{s+2d}
$$
and the boundary conditions
\begin{eqnarray*}
& & (u_j-v_j)(0,t) = \left(R(u-v)\right)_j(t),\quad 1\le j\le m,\quad t\ge s,
\\
& & (u_j-v_j)(1,t) =  \left(R(u-v)\right)_j(t), \quad m<j\le n,\quad t\ge s,
\end{eqnarray*}
This implies the  operator equality
\begin{eqnarray}
\label{rep_(u^l-v^l)}
(u-v)\big|_{\overline\Pi_{s+2d}}=(SR)(u-v)+D(u-v)+F\left(\tilde a\partial_xv\right)+
F\bigl(\tilde bv\bigr),
\end{eqnarray}
where the operators $S$ and $R$ are given  by \reff{S}, \reff{eq:R},
respectively, and $D, F: BC\left(\overline\Pi_s;\R^n\right) \to
BC\left(\overline\Pi_s;\R^n\right)$ are  linear bounded operators defined by
\begin{eqnarray*}
\left(D w\right)_j(x,t) & = &
-\int_{x_j(x,t)}^x d_j(\xi,x,t)\sum_{k=1\atop k\not=j}^nb_{jk}(\xi,\om_j(\xi))
w_k(\xi,\om_j(\xi))\,d\xi,\quad j\le n,\\
\left(F f\right)_j(x,t)&=&\int_{x_j(x,t)}^x d_j(\xi,x,t)f_j(\xi,\om_j(\xi))d\xi,\quad j\le n.
\end{eqnarray*}
Since $u-v$ occurs in both sides of \reff{rep_(u^l-v^l)}, this equation can be iterated. Note that $D$
operates with $u-v$ on a different (shifted) domain. Hence, such iteration is possible only on a subdomain
of $\overline\Pi_{s+2d}$. Specifically, $n$ iterations are possible on $\overline\Pi_{s+3d}$ and, doing so, on the first step we obtain
\begin{eqnarray}
(u-v)\big|_{\overline\Pi_{s+3d}}
=(SR)^2(u-v)+(I+SR)D(u-v)
+(I+SR)F\left(\tilde a\partial_xv\right)+(I+SR)F\bigl(\tilde bv\bigr).\nonumber
\end{eqnarray}
Iterating this, that is, substituting \reff{rep_(u^l-v^l)} into the last equation once and once again, in the $n$-th
step we meet the property ${\bf(H3)^\prime}$, resulting in the identity
$$
(SR)^{n}(u-v)\equiv 0.
$$
Consequently, we get 
$$
(u-v)\big|_{\overline\Pi_{s+3d}}=\sum_{i=0}^{n-1}(SR)^{i}D(u-v)+
\sum_{i=0}^{n-1}(SR)^{i}F\left(\tilde a\partial_xv\right)+
\sum_{i=0}^{n-1}(SR)^{i}F\bigl(\tilde bv\bigr).
$$
This gives us the desired formula 
\beq\label{rep_(u-v)_smooth}
(u-v)(x,s+3d)=\left[\sum_{i=0}^{n-1}(SR)^{i}D(u-v)+
\sum_{i=0}^{n-1}(SR)^{i}F\left(\tilde a\partial_xv\right)+
\sum_{i=0}^{n-1}(SR)^{i}F\bigl(\tilde bv\bigr)\right](x,s+3d).
\ee

To prove the estimate \reff{main_estim}, we derive appropriate  smallness bounds for
each of the three summands in the right hand side of  \reff{rep_(u-v)_smooth} separately.

\vskip2mm
{\it Step 2.
Obtaining an upper bound of the type $\be(\eps)\|\vphi\|_{L^2\left((0,1);\R^n\right)}$ for the
second and the third summands  in \reff{rep_(u-v)_smooth}.
}
Given $s<\tau<\infty$, denote
$$
\Pi_{s}^\tau = \{(x,t)\,:\,0<x<1, s<t<\tau\}.
$$
Since the equality \reff{rep_(u-v)_smooth}  is considered at
$t=s+3d$, 
 the operator $F$ in the right-hand side of \reff{rep_(u-v)_smooth} 
 operates with the functions $\d_xv$ 
 on $\overline\Pi_{s+2d}^{s+3d}$, which allows us to use 
 the smoothing estimate \reff{apr_vv1}. More precisely, we apply the Cauchy-Schwarz inequality
 to $F\left(\tilde a\partial_xv\right)$ and  then use the  estimate \reff{apr_vv1}.  The needed bound for the
second  summand   immediately follows  from \reff{apr_vv1}, the boundedness of the operators $S$ and $R$, 
and the smallness of 
$\tilde a$ and $\tilde b$. The desired bound for  the third summand is a 
simple consequence of the smallness of 
$\tilde a$ and~ $\tilde b$.

\vskip2mm

To estimate  the first summand in the right-hand side of
\reff{rep_(u-v)_smooth}, it suffices  to derive a smallness bound for $D(u-v)$.

\vskip2mm
{\it Step 3.
Derivation of an operator equation for $D(u-v)$.
}
The continuous solutions  $u$ and $v$ on $\overline\Pi_s^{s+3d}$ satisfy the operator equations
$$
u=Qu+Du,\quad v=\tilde Qv+\tilde D v,
$$
where
 the operators $Q,\tilde Q, D,$ and $\tilde D$ are restricted to the subspace of
$C\bigl(\overline\Pi_s^{s+3d};\R^n\bigr)$ of functions satisfying the initial
condition (\ref{eq:in}).
Note that the operators  $Q,\tilde Q, D,$ and $\tilde D$ map
$C\bigl(\overline\Pi_s^{s+3d};\R^n\bigr)$ into itself.
Thus, for the difference $u-v$ we have
\begin{eqnarray}
\label{rep_(u-v)0}
u-v=Q(u-v)+\bigl(Q-\tilde Q\bigr)v+D(u-v)+\bigl(D-\tilde D\bigr)v,
\end{eqnarray}
hence
\begin{eqnarray}
\label{rep_(u-v)1}
D(u-v)=DQ(u-v)+D\bigl(Q-\tilde Q\bigr)v+D^2(u-v)+D\bigl(D-\tilde D\bigr)v.
\end{eqnarray}
Substitute \reff{rep_(u-v)0} into the first summand in the right-hand side of \reff{rep_(u-v)1} and
 rewrite the last equation with respect to the new variable
$
w=D(u-v).
$
We get
\begin{eqnarray}
\label{rep_(u-v)2}
w=DQ^2(u-v)+D(I+Q)\bigl(Q-\tilde Q\bigr)v
+D(I+Q)w+D(I+Q)\bigl(D-\tilde D\bigr)v.
\end{eqnarray}
Continuing in this fashion (again substituting \reff{rep_(u-v)0} into the first summand
in the right-hand side of \reff{rep_(u-v)2}), 
in the $n$-th step  we arrive at the
 formula
\begin{eqnarray*}
w=DQ^n(u-v)+D\sum_{i=0}^{n-1}Q^i\bigl(Q-\tilde Q\bigr)v
+D\sum_{i=0}^{n-1}Q^iw+D\sum_{i=0}^{n-1}Q^i\bigl(D-\tilde D\bigr)v.
\end{eqnarray*}
Furthermore, combining  the condition ${\bf(H3)^\prime}$ and the fact that $(u-v)(\cdot,s)\equiv 0$ on $[0,1]$,
we conclude that
 $\left[Q^{n}(u-v)\right](x,t)\equiv 0$ on $\overline\Pi_s^{s+3d}$.
The resulting equation for $w$ restricted to $\overline\Pi_s^{s+3d}$
reads
\begin{eqnarray}
\label{rep_(u-v)4}
w=D\sum_{i=0}^{n-1}Q^i\bigl(Q-\tilde Q\bigr)v+D\sum_{i=0}^{n-1}Q^i\bigl(D-\tilde D\bigr)v
+D\sum_{i=0}^{n-1}Q^iw.
\end{eqnarray}

\vskip2mm
{\it Step 4.
Obtaining an upper  bound of the type $\be(\eps)\|\vphi\|_{L^2\left((0,1);\R^n\right)}$ for  $w=D(u-v)$.
}
Next we prove that there exists a function
$\beta_0:[0,1]\to\R$ with $\beta_0(\eps)\to 0$ as $\eps\to 0$, for which we have the estimate
\begin{eqnarray}
\label{apr_w}
\max\limits_{s\le t\le s+3d}\|w(\cdot,t)\|_{L^2\left((0,1);\R^n\right)}\le
\be_0(\eps)
\|\vphi\|_{L^2\left((0,1);\R^n\right)},
\end{eqnarray}
being uniform in   $s\in\R$, $\vphi\in C_0^1([0,1];\R^n)$, and $\tilde a,\tilde b\in BC^1\left(\overline\Pi;\M_n\right)$
with
$\|\tilde a\|_{BC^1(\Pi;\M_n)}\le \eps$ and
$\|\tilde b\|_{BC(\Pi;\M_n)} \le \eps$.

By technical reasons, we rewrite  the integral operator $D$
in the following equivalent form, obtained using integration along characteristic curves  in $t$
(rather than in $x$)
\begin{eqnarray*}
\left(D w\right)_j(x,t) & = &
-\int_{t_j(x,t)}^t  \exp \int_t^\tau
b_{jj}(\sigma_j(\eta),\eta)\,d\eta
\sum_{k\not=j}b_{jk}(\sigma_j(\tau),\tau)
w_k(\sigma_j(\tau),\tau)\,d\tau,
\end{eqnarray*}
where 
$$
\tau\in\R \mapsto \sigma_j(\tau)=\sigma_j(\tau,x,t)\in[0,1]
$$ 
is the inverse form of
the $j$-th characteristic of \reff{eq:1}
passing through the point $(x,t)\in\overline\Pi$, $t_j(x,t)$ is the minimum value of $\tau$
at which the characteristic $\tau=\sigma_j(\tau,x,t)$ reaches $\d\Pi_s$.
The function $\sigma_j(\tau)$ is the solution to the initial value problem
$$\partial_\tau \sigma_j(\tau,x,t) = a_j(\sigma_j(\tau,x,t),\tau), \
\sigma_j(t,x,t) = x.$$
Therefore, the  estimate \reff{apr_w}  follows from the Gronwall's inequality applied to \reff{rep_(u-v)4},
provided the first two summands satisfy an upper  bound of the type $\be(\eps)\|\vphi\|_{L^2\left((0,1);\R^n\right)}$.

The rest of the proof consists in deriving  the desired upper bound for the first two summands
in the right-hand side of \reff{rep_(u-v)4}. In Steps 5--8 we get the desired bound for the second
summand, while in Step 9 we get it for the first summand.
\vskip2mm

{\it Step 5.
Derivation of  a representation formula for the
second summand  in  \reff{rep_(u-v)4}.
}
Remark that the main technicalities appear already in  the case $i=0$ and the proof for  $i\ge 1$ uses a  similar
argument. Hence, let $i=0$  and
 estimate  the summand $D(D-\tilde D)v$.

In what follows, we will use the following notation.
 The $j$-th characteristic of \reff{eq:1p}
passing through the point $(x,t)\in\overline\Pi_s$ is defined
as the solution $\xi\in [0,1] \mapsto \tilde\om_j(\xi)=\tilde\om_j(\xi,x,t)\in \R$ of the initial value problem
\beq\label{char1}
\partial_\xi\tilde\om_j(\xi, x,t)=\frac{1}{[a_j+\tilde a_j](\xi,\tilde\om_j(\xi, x,t))},\;\;
\tilde\om_j(x,x,t)=t.
\ee
Write
\begin{eqnarray*}
\tilde c_j(\xi,x,t)=\exp \int_x^\xi
\left[\frac{b_{jj}+\tilde b_{jj}}{a_{j}+\tilde a_{j}}\right](\eta,\tilde\om_j(\eta))\,d\eta,\quad
\tilde d_j(\xi,x,t)=\frac{\tilde c_j(\xi,x,t)}{[a_j+\tilde a_j](\xi,\tilde\om_j(\xi))}.
\end{eqnarray*}
Introduce  the linear bounded operator $\tilde D: BC\left(\overline\Pi_s;\R^n\right) \to
BC\left(\overline\Pi_s;\R^n\right)$ and the affine bounded operator
$\tilde Q: BC\left(\overline\Pi_s;\R^n\right)\mapsto BC\left(\overline\Pi_s;\R^n\right)$   by
\begin{eqnarray*}
\left(\tilde D w\right)_j(x,t) & = &
-\int_{\tilde x_j(x,t)}^x \tilde d_j(\xi,x,t)\sum_{k\not=j}
\left[b_{jk}+\tilde b_{jk}\right](\xi,\tilde \om_j(\xi))
w_k(\xi,\tilde \om_j(\xi))\,d\xi,\\
\left(\tilde Qu\right)_j(x,t) &=&
\left\{\begin{array}{lcl}
\tilde c_j(\tilde x_j(x,t),x,t)(Ru)_j(\tilde \om_j(\tilde x_j(x,t),x,t)) & \mbox{if}&  \tilde x_j(x,t)\notin(0,1), \\
\tilde c_j(\tilde x_j(x,t),x,t)\vphi_j(\tilde x_j(x,t))      & \mbox{if}&  \tilde x_j(x,t)\in(0,1),
\end{array}\right.
\end{eqnarray*}
where $\tilde x_j(x,t)$ denotes the abscissa of the point with the smallest ordinate,
at which the characteristic curve $\tau=\tilde\om_j(\xi,x,t)$ reaches the
boundary of $\Pi_s$.
Set
$$
\begin{array}{ll}
d_{jki}(\xi,\eta,x,t)=d_j(\xi,x,t) d_k(\eta,\xi,\om_j(\xi))b_{jk}(\xi,\om_j(\xi))
b_{ki}(\eta,\om_k(\eta,\xi,\om_j(\xi))),\\ [2mm]
\tilde d_{jki}(\xi,\eta,x,t)=d_j(\xi,x,t)\tilde d_k(\eta,\xi,\om_j(\xi))b_{jk}(\xi,\om_j(\xi))
\left[b_{ki}+\tilde b_{ki}\right](\eta,\tilde \om_k(\eta,\xi,\om_j(\xi))).
\end{array}
$$
Then we have
\begin{equation}\label{D11}
\begin{array}{ll}
\displaystyle\Bigl[\bigl(D^2-D\tilde D\bigr)v\Bigr]_j(x,t)
=\sum_{k\not=j}\sum_{i\not=k}
\int_{x_j(x,t)}^x \int_{x_k(\xi,\om_j(\xi))}^\xi d_{jki}(\xi,\eta,x,t)
v_i(\eta,\om_k(\eta,\xi,\om_j(\xi)))\, d \eta d \xi\\
\displaystyle-\sum_{k\not=j}\sum_{i\not=k}
\int_{x_j(x,t)}^x \int_{\tilde x_k(\xi,\om_j(\xi))}^\xi\tilde  d_{jki}(\xi,\eta,x,t)
v_i(\eta,\tilde \om_k(\eta,\xi,\om_j(\xi)))\, d \eta d \xi\\
\displaystyle
=\sum_{k\not=j}\sum_{i\not=k}
\int_{x_j(x,t)}^x \int_{x_k(\xi,\om_j(\xi))}^{\tilde x_k(\xi,\om_j(\xi))}\tilde  d_{jki}(\xi,\eta,x,t)
v_i(\eta,\tilde \om_k(\eta,\xi,\om_j(\xi)))\, d \eta d \xi\\
\displaystyle
+\sum_{k\not=j}\sum_{i\not=k}
\int_{x_j(x,t)}^x \int_{x_k(\xi,\om_j(\xi))}^\xi\left(d_{jki}(\xi,\eta,x,t)-\tilde
 d_{jki}(\xi,\eta,x,t)\right)v_i(\eta,\tilde \om_k(\eta,\xi,\om_j(\xi)))\, d \eta d \xi
\\
+
\displaystyle\sum_{k\not=j}\sum_{i\not=k}
\displaystyle
\int_{x_j(x,t)}^x \int_{x_k(\xi,\om_j(\xi))}^\xi d_{jki}(\xi,\eta,x,t)
\Bigl(v_i(\eta,\om_k(\eta,\xi,\om_j(\xi)))
-v_i(\eta,\tilde \om_k(\eta,\xi,\om_j(\xi)))
\Bigr)
\, d \eta d \xi.
\end{array}
\end{equation}
Let us estimate each of the three summands in the right hand side separately.
\vskip2mm
{\it Step 6.
Obtaining some technical inequalities.
}
Due to the regularity and the boundedness assumptions on the coefficients $a$, $\tilde a$,  $b$,  and $\tilde b$,
we have
\beq\label{al2}
\begin{array}{ll}
\displaystyle\max_{\xi,\eta\in[0,1]}\left\|\tilde \om_k(\eta,\xi,\om_j(\xi,\cdot,\cdot))-\om_k(\eta,\xi,\om_j(\xi,\cdot,\cdot))
\right\|_{C\left(\overline\Pi_s^{s+3d}\right)}\le\tilde\be(\eps),\\
\displaystyle\max_{\xi\in[0,1]}
\|\tilde x_k(\xi,\om_j(\xi,\cdot,\cdot))-x_k(\xi,\om_j(\xi,\cdot,\cdot))\|_{C\left(\overline\Pi_s^{s+3d}\right)}=
\|\tilde x_k-x_k\|_{C\left(\overline\Pi_s^{s+3d}\right)}\le\tilde\be(\eps),\\
\displaystyle\max_{\xi,\eta\in[0,1]}
\|\tilde  d_{jki}(\xi,\eta,\cdot,\cdot)-d_{jki}(\xi,\eta,\cdot,\cdot)\|_{C\left(\overline\Pi_s^{s+3d}\right)}\le\tilde\be(\eps),
\\
\displaystyle\max_{\xi,\eta\in[0,1]}\left\|\frac{d}{d\xi}\left[\tilde \om_k(\eta,\xi,\om_j(\xi,\cdot,\cdot)-
\om_k(\eta,\xi,\om_j(\xi,\cdot,\cdot))\right]
\right\|_{C\left(\overline\Pi_s^{s+3d}\right)}\le\tilde\be(\eps), \\
\displaystyle\max_{\xi\in[0,1]}
\left\|\frac{d}{d\xi}\left[\tilde x_k(\xi,\om_j(\xi,\cdot,\cdot))-x_k(\xi,\om_j(\xi,\cdot,\cdot))\right]\right\|_{C\left(\overline\Pi_s^{s+3d}\right)}
\le\tilde\be(\eps)
\end{array}
\ee
for all $s\in\R$, for all  $\tilde a$ and $\tilde b$ with
$\|\tilde a\|_{BC^1(\Pi;\M_n)} \le \eps$ and
$\|\tilde b\|_{BC(\Pi;\M_n)} \le \eps$, for all $j,k\le n$, and for a function $\tilde\be:[0,1]\to\R$
 approaching zero as $\eps\to 0$.
In order to prove $(\ref{al2})_1$ we use  the equations
\reff{char} and \reff{char1} and obtain
\begin{eqnarray*}
\frac{d}{d\eta}(\om_k(\eta) - \tilde\om_k(\eta)) =
  \frac{a_k(\eta,\tilde\om_k(\eta)) - a_k(\eta,\om_k(\eta)) + \tilde a_k(\eta,\tilde\om_k(\eta))}
 {a_k(\eta,\om_k(\eta))(a_k(\eta,\tilde\om_k(\eta))+\tilde a_k(\eta,\tilde\om_k(\eta)))},\quad
\om_k(x) - \tilde\om_k(x)=0.
\end{eqnarray*}
Application of \reff{eq:h11} gives
\begin{eqnarray*}
 |\om_k(\eta) - \tilde\om_k(\eta)| \le \frac{1}{\Lambda_0^2}\left|\int_x^\eta
\left( \|a_k\|_{BC^1(\Pi)}|\om_k(\eta_1) - \tilde\om_k(\eta_1)|
 + \|\tilde a_k\|_{BC(\Pi)}\right)\,d\eta_1\right|.
\end{eqnarray*}
The Gronwall's inequality yields
\begin{eqnarray*}
 |\om_k(\eta) - \tilde\om_k(\eta)| \le \frac{1}{\Lambda_0^2} 
 \|\tilde a_k\|_{BC(\Pi)} \exp\left\{\frac{|\eta-x|}{\Lambda_0^2} \|a_k\|_{BC^1(\Pi)}\right\},
\end{eqnarray*}
implying  $(\ref{al2})_1$.

To derive  $(\ref{al2})_2$, we proceed similarly, but now we consider the initial value problem for the difference
$\tilde\sigma_k(\tau,x,t) - \sigma_k(\tau,x,t)$, namely
\begin{eqnarray*}
 \frac{d}{d\tau}(\tilde\sigma_k(\tau) - \sigma_k(\tau)){d\tau} = a_k(\tilde\sigma_k(\tau), \tau) +
\tilde a_k(\tilde\sigma_k(\tau), \tau) -
 a_k(\sigma_k(\tau),\tau),\quad \tilde \sigma_k(t) - \sigma_k(t)=0.
\end{eqnarray*}
It remains to recall that, if $0<x_k(x,t)<1$, then  $x_k(x,t)=\sigma_k(s,x,t)$.

The estimate $(\ref{al2})_3$ follows directly from $(\ref{al2})_1$ and the smallness of $\tilde a$
and $\tilde b$.

To prove $(\ref{al2})_4$, we use the identities
\beq\label{dx}
\d_x\tilde\om_k(\xi,x,t)=-\frac{1}{a_k(x,t)+\tilde a_k(x,t)} \exp \int_\xi^x
\frac{[\d_ta_k+\d_t\tilde a_k](\eta,\tilde\om_k(\eta))}{[a_k+\tilde a_k]^2(\eta,\tilde\om_k(\eta))}\, d\eta,
\ee
\beq\label{dt}
\d_t\tilde\om_k(\xi,x,t)= \exp \int_\xi^x
\frac{[\d_ta_k+\d_t\tilde a_k](\eta,\tilde\om_k(\eta))}{[a_k+\tilde a_k]^2(\eta,\tilde\om_k(\eta))}\, d\eta
\ee
and do the following calculations:
\begin{eqnarray*}
 & & \frac{d}{d\xi}(\omega_k(\eta,\xi,\omega_j(\xi)) -\tilde\omega_k(\eta,\xi,\omega_j(\xi))) =
 \partial_2[\omega_k - \tilde\omega_k](\eta,\xi,\omega_j(\xi)) \\
 & & + \frac{1}{a_j(\xi,\omega_j(\xi))}\partial_3 [\omega_k - \tilde\omega_k](\eta,\xi,\omega_j(\xi))   \\
 & & = \frac{1}{[a_k + \tilde a_k](\xi,\omega_j(\xi))} \exp\int_\xi^x \frac{\partial_2 [a_k +\tilde a_k](\eta_1,\omega_k(\eta_1,\xi,\omega_j(\xi)))}
 {[a_k + \tilde a_k]^2(\eta_1,\omega_k(\eta_1,\xi,\omega_j(\xi)))} d\eta_1  \\
 & & - \frac{1}{a_k(\xi,\omega_j(\xi))} \exp\int_\xi^x \frac{\partial_2 a_k(\eta_1,\omega_k(\eta_1,\xi,\omega_j(\xi)))}
 {a_k^2(\eta_1,\omega_k(\eta_1,\xi,\omega_j(\xi)))} d\eta_1  \\
 & & - \frac{1}{a_j(\xi,\omega_j(\xi))} \Biggl( \exp\int_\xi^x \frac{\partial_2 [a_k +\tilde a_k](\eta_1,\omega_k(\eta_1,\xi,\omega_j(\xi)))}
 {[a_k + \tilde a_k]^2(\eta_1,\omega_k(\eta_1,\xi,\omega_j(\xi)))} d\eta_1  \\
 & & - \exp\int_\xi^x \frac{\partial_2 a_k(\eta_1,\omega_k(\eta_1,\xi,\omega_j(\xi)))}
 {a_k^2(\eta_1,\omega_k(\eta_1,\xi,\omega_j(\xi)))} d\eta_1
 \Biggl).
 \end{eqnarray*}
 Here and in what follows $\d_i$  denotes the partial derivative with respect to the $i$-th  argument.
The estimate $(\ref{al2})_4$ is now an easy consequence of the inequality
$\|\tilde a\|_{BC(\Pi;\M_n)} + \|\partial_t \tilde a\|_{BC(\Pi;\M_n)} \le \varepsilon$.

Finally, to prove $(\ref{al2})_5$, we take into account the equalities
\begin{eqnarray*}
& & \partial_x \sigma_k(\tau,x,t) = \exp \int_t^\tau \partial_1 a_k( \sigma_k(\eta,x,t),\eta)d\eta, \\
& & \partial_t \sigma_k(\tau,x,t) = - a_k(x,t)\exp \int_t^\tau \partial_1 a_k( \sigma_k(\eta,x,t),\eta)d\eta,
 \end{eqnarray*}
 which yields
\begin{eqnarray*}
  & & \frac{d}{d\xi} \left(\sigma_k(\tau,\xi,\omega_j(\xi)) -\tilde\sigma_k(\tau,\xi,\omega_j(\xi))\right) =
   \partial_2[\sigma_k - \tilde\sigma_k](\tau,\xi,\omega_j(\xi))  \\
 & & + \frac{1}{a_j(\xi,\omega_j(\xi))}\partial_3 [\sigma_k - \tilde\sigma_k](\tau,\xi,\omega_j(\xi))   \\
 & & = \exp \int_t^\tau \partial_1 a_k(\sigma_k(\eta_1,\xi,\omega_j(\xi)),\eta_1)d\eta_1 -
 \exp \int_t^\tau \partial_1[a_k +\tilde a_k]( \sigma_k(\eta_1,\xi,\omega_j(\xi)),\eta_1)d\eta_1  \\
 & & - \frac{1}{a_j(\xi,\omega_j(\xi))}\Biggl[ a_k(\xi,\omega_j(\xi))
 \exp \int_t^\tau \partial_1 a_k( \sigma_k(\eta_1,\xi,\omega_j(\xi)),\eta_1)d\eta_1  \\
 & & - [a_k + \tilde a_k](\xi,\omega_j(\xi))  \exp \int_t^\tau \partial_1[a_k +\tilde a_k]( \sigma_k(\eta_1,\xi,\omega_j(\xi)),\eta_1)d\eta_1 \Biggl].
\end{eqnarray*}
Similarly to the above, the estimate $(\ref{al2})_5$ now follows directly from the bound
$\|\tilde a\|_{BC(\Pi;\M_n)} + \|\partial_x \tilde a\|_{BC(\Pi;\M_n)} \le \varepsilon$.

 Since $\tilde\be(\eps)\to 0$ as  $\eps\to 0$, below
 we  suppose that $\tilde\be(\eps)<1$.

\vspace{2mm}
{\it Step 7.
Obtaining an upper bound of the type $\be(\eps)\|\vphi\|_{L^2\left((0,1);\R^n\right)}$ for the
first and the second summands  in  the right-hand side of \reff{D11}.
}
 For the integrals in the first summand  we
use $(\ref{al2})_1$ and the Cauchy-Schwarz inequality, obtaining
\begin{eqnarray}
\lefteqn{
\left|
\int_{x_j(x,t)}^x \int_{\tilde x_k(\xi,\om_j(\xi))}^{x_k(\xi,\om_j(\xi))}\tilde  d_{jki}(\xi,\eta,x,t)
v_i(\eta,\tilde \om_k(\eta,\xi,\om_j(\xi)))\, d \eta d \xi
\right|
}\nonumber\\
&&
\le\max\limits_{\xi,\eta,x\in[0,1]}\sup\limits_{t\in\R}|\tilde d_{jki}(\xi,\eta,x,t)|
\max\limits_{y\in[0,1-\tilde\be(\eps)]}\int_{x_j(x,t)}^x\int_{y}^{y+\tilde\be(\eps)}
|v_i(\eta,\tilde\om_k(\eta,\xi,\om_j(\xi)))|\, d \eta d \xi\nonumber\\
&&=\max\limits_{\xi,\eta,x\in[0,1]}\sup\limits_{t\in\R}|\tilde d_{jki}(\xi,\eta,x,t)| \nonumber\\
& & \times \max\limits_{y\in[0,1-\tilde\be(\eps)]}\left\{
\int_{y}^{y+\tilde\be(\eps)} \int_{x_j(x,t)}^x|v_i(\eta,\tilde\om_k(\eta,\xi,\om_j(\xi)))|\, d \xi d \eta\right\}.\label{1st}
\end{eqnarray}
For a fixed  $\eta$, let us change the variables
\beq\label{change}
\xi\mapsto\theta=\tilde\om_k(\eta,\xi,\om_j(\xi)).
\ee
Taking into  account
the equalities \reff{dx} and \reff{dt},
from \reff{change} we get
\begin{eqnarray}\label{theta1}
d\theta& = & \Bigl(\d_2\tilde\om_k(\eta,\xi,\om_j(\xi))+\d_3\tilde\om_k(\eta,\xi,\om_j(\xi))\d_\xi\om_j(\xi)\Bigl)d\xi\nonumber\\
\displaystyle
& = &\left[\frac{a_k+\tilde a_k-a_j}{a_j(a_k+\tilde a_k)}\right](\xi,\om_j(\xi))
\d_3\tilde \om_k(\eta,\xi,\om_j(\xi))\, d\xi.
\end{eqnarray}
As it follows
from \reff{theta1}, the change of variables \reff{change} is non-degenerate  for all $\xi,x\in[0,1]$  and $t\in[s,s+3d]$ whenever
 $\left[a_k+\tilde a_k\right](\xi,\om_j(\xi))-a_j(\xi,\om_j(\xi))\ne 0$. Remark that the last condition is true due to the assumption
$\bf(H1)$  and the choice of $\eps_0$. Denote the inverse of \reff{change} by
  $\tilde x(\theta)=\tilde x(\theta,\eta,x,t)$. One can see that  $\tilde x(\theta,\eta,x,t)$ is continuous in all its arguments.
Therefore, changing the variables according to \reff{change},
the double integral in the right-hand side of \reff{1st} reads
\begin{eqnarray*}
& & \int_{y}^{y+\tilde\be(\eps)}\int_{\tilde\om_k(\eta,x_j(x,t),\om_j(x_j(x,t)))}^{\tilde\om_k(\eta)}
\Biggl|\left[\frac{a_j(a_k+\tilde a_k)}{a_k+\tilde a_k-a_j}\right](\tilde x(\theta),\om_j(\tilde x(\theta))) \\
& & \times\left[\d_3\tilde \om_k(\eta,\tilde x(\theta),\om_j(\tilde x(\theta)))\right]^{-1}v_i(\eta,\theta)\Biggl|\,d\theta d\eta \\
& & \le C_1
\max\limits_{\theta\in[s,s+3d]}\int_{y}^{y+\tilde\be(\eps)}|v_i(\eta,\theta)|\,d\eta\le
C_1 \tilde\be(\eps)\max\limits_{\theta\in[s,s+3d]}\left(\int_{y}^{y+\tilde\be(\eps)}|v_i(\eta,\theta)|^2\,d\eta\right)^{1/2} \\
& & \le
KC_1 \tilde\be(\eps)e^{3d\nu}\|\vphi\|_{L^2\left((0,1);\R^n\right)}\le
 \be_1(\eps)\|\vphi\|_{L^2\left((0,1);\R^n\right)},
\end{eqnarray*}
where
$$
\begin{array}{cc}
C_1=\max\limits_{\eta,x\in[0,1]}\sup\limits_{t\in\R}
\left|
\tilde\om_k(\eta)-\tilde\om_k(\eta,x_j(x,t),\om_j(x_j(x,t)))
\right|\\\times
\displaystyle\max\limits_{\theta,\eta,x\in[0,1]}\sup\limits_{t\in\R}
\left|\left[\frac{a_j(a_k+\tilde a_k)}{a_k+\tilde a_k-a_j}\right](\tilde x(\theta),\om_j(\tilde x(\theta)))
\left[\d_3\tilde \om_k(\eta,\tilde x(\theta),\om_j(\tilde x(\theta)))\right]^{-1}\right|
\end{array}
$$
and the function $\be_1:[0,1]\to\R$
 approaches zero as $\eps\to 0$.
 Here we used
the assumption (\ref{eq:h11}) and the estimate \reff{k3}
about the exponential boundedness of the evolution operator. The desired estimate for the first summand
 in \reff{D11} is derived.

Similar estimate for the second summand  in \reff{D11} immediately follows from
the assumption (\ref{eq:h11}) and the estimates \reff{al2}.

\vspace{2mm}
{\it Step 8.
Obtaining an upper bound of the type $\be(\eps)\|\vphi\|_{L^2\left((0,1);\R^n\right)}$ for the
third  summand  in  the right-hand side of \reff{D11}.
} 
Fix arbitrary $1\le j,k\le m$
(for the other $j, k$ we proceed similarly)
and use the mean value theorem and the estimates (\ref{al2}).
This results in the following representation of the third  summand,  which will be denoted  by  $I_1(x,t)$:
\begin{eqnarray}
I_1(x,t)=\lefteqn{
\int_{x_j(x,t)}^x \int_{x_k(\xi,\om_j(\xi))}^\xi d_{jki}(\xi,\eta,x,t)
\left( \om_k(\eta,\xi,\om_j(\xi))-\tilde\om_k(\eta,\xi,\om_j(\xi))\right)
}\nonumber\\
&&\times\int_0^1
\d_2v_i\Bigl(\eta,\gamma \om_k(\eta,\xi,\om_j(\xi))+(1-\gamma)\tilde\om_k(\eta,\xi,\om_j(\xi))\Bigl)
\, d\gamma d \eta d \xi. \label{ik2}
\end{eqnarray}
Using the notation
\begin{eqnarray*}
 & & \rho(\xi,\eta,x,t,\gamma) = \gamma\left[\frac{a_k-a_j}{a_j a_k}\right](\xi,\omega_j(\xi))\partial_3\omega_k(\eta,\xi,\omega_j(\xi))  \nonumber\\
 & & + (1-\gamma)\left[\frac{a_k +\tilde a_k -a_j}{a_j(a_k + \tilde a_k)}\right](\xi,\omega_j(\xi))\partial_3\tilde\omega_k(\eta,\xi,\omega_j(\xi)),
\end{eqnarray*}
we have
\begin{eqnarray}
 & & \frac{d}{d\xi}v_i\left(\eta, \gamma\omega_k(\eta,\xi,\omega_j(\xi)) + (1 - \gamma)\tilde\omega_k(\eta,\xi,\omega_j(\xi))\right) \nonumber\\
 & & = \partial_2 v_i\Bigl(\eta, \gamma\omega_k(\eta,\xi,\omega_j(\xi)) + (1 - \gamma)\tilde\omega_k(\eta,\xi,\omega_j(\xi))\Bigl)
 \rho(\xi,\eta,x,t,\gamma). \label{ik3}
\end{eqnarray}
Remark that $\rho(\xi,\eta,x,t,\gamma) \not= 0$ for all $\xi,\eta,x \in [0,1], \ t \in [s,s+3d],$ and $\gamma \in [0,1],$
since  our  assumptions imply that 
$(a_k + \tilde a_k - a_j)(\xi, \omega_j(\xi)) \not= 0$ for all $\|\tilde a_k\|_{BC(\Pi)} \le \varepsilon_0.$
Note also that $\partial_3\omega_k$ and  $\partial_3\tilde\omega_k$ are strictly positive, see (\ref{dt}).
On the account of (\ref{ik3}), the expression (\ref{ik2}) can be rewritten as follows:
\begin{eqnarray}
& &  I_1(x,t) = \int_{x_j(x,t)}^x \frac{d}{d\xi} \Biggl(\int_{x_k(\xi,\om_j(\xi))}^\xi d_{jki}(\xi,\eta,x,t)
\left(\om_k(\eta,\xi,\om_j(\xi))-\tilde \om_k(\eta,\xi,\om_j(\xi))\right)
\nonumber\\
& & \times\int_0^1 \rho^{-1}(\xi,\eta,x,t,\gamma)
v_i\Bigl(\eta,\gamma\om_k(\eta,\xi,\om_j(\xi))+(1-\gamma)\tilde\om_k(\eta,\xi,\om_j(\xi))\Bigl)
\, d\gamma d \eta \Biggl) d \xi \nonumber\\
& & + \int_{x_j(\xi,t)}^x \frac{dx_k(\xi,\omega_j(\xi))}{d\xi} d_{jki}(\xi, x_k(\xi,\omega_j(\xi),x,t)
\left[\omega_k - \tilde\omega_k\right](x_k(\xi,\omega_j(\xi)),\xi,\omega_j(\xi))  \nonumber\\
& & \times \int_0^1 \rho^{-1}(\xi, x_k(\xi,\omega_j(\xi)),x,t,\gamma) v_i\Bigl(x_k(\xi,\omega_j(\xi)), [\gamma\omega_k + (1-\gamma)\tilde\omega_k]
(x_k(\xi,\omega_j(\xi)),\xi,\omega_j(\xi))\Bigl)d\gamma d\xi  \nonumber\\
& & - \int_{x_j(\xi,t)}^x \int_{x_k(\xi,\omega_j(\xi))}^\xi \frac{d}{d\xi} \Bigl[ d_{jki}(\xi,\eta,x,t)
\left(\omega_k(\eta,\xi,\omega_j(\xi)) - \tilde\omega_k(\eta,\xi,\omega_j(\xi))\right)  \nonumber\\
& & \times \rho^{-1}(\xi,\eta,x,t,\gamma)\Bigl] v_i\Bigl(\eta,\gamma\omega_k(\eta,\xi,\omega_j(\xi)) + (1-\gamma)\tilde\omega_k(\eta,\xi,\omega_j(\xi))\Bigl)
d\gamma d\eta d\xi. \label{ik4}
\end{eqnarray}

Denote by $x_{jk}(\theta,x,t)$ the $x$-coordinate of the 
point where  the characteristics
$\omega_j(\xi,x,t)$ and $\omega_k(\xi,\theta,s)$ intersect (if they do),  
that is
\beq\label{2.60}
\omega_j(x_{jk}(\theta,x,t),x,t)=\omega_k(x_{jk}(\theta,x,t),\theta,s).
\ee
Suppose  for definiteness that $a_j(x,t) > a_k(x,t)$ (the case of $a_j(x,t) < a_k(x,t)$ is similar).
Since $x_k(\xi,\omega_j(\xi)) = 0$ for all $\xi \in [x_j(x,t), x_{jk}(0,x,t)]$,
the integral over the interval $[x_j(x,t), x_{jk}(0,x,t)]$ in the  second summand of (\ref{ik4}) disappears. Furthermore, if $x_{jk}(0,x,t)\notin[x_j(x,t),x]$,
then evidently the integral over $[x_{j}(x,t),x]$  in this summand disappears.
We therefore need to estimate
the second summand in (\ref{ik4})  whenever $x_{jk}(0,x,t)\in[x_j(x,t),x]$. If this is the case,  then the
second summand reads
\begin{eqnarray}
 & & \int_{x_{jk}(0,x,t)}^x\frac{dx_k(\xi,\omega_j(\xi))}{d\xi}d_{jki}(\xi,x_k(\xi,\omega_j(\xi)),x,t)  \nonumber\\
 & & \times[\omega_k - \tilde\omega_k](x_k(\xi,\omega_j(\xi)),\xi,\omega_j(\xi)) \nonumber\\
 & & \times \int_0^1 \rho^{-1}(\xi,x_k(\xi,\omega_j(\xi)),x,t,\gamma) \nonumber\\
& &  \times v_i\Bigl(x_k(\xi,\omega_j(\xi)), [\gamma\omega_k
 + (1-\gamma)\tilde\omega_k]
(x_k(\xi,\omega_j(\xi)),\xi,\omega_j(\xi))\Bigl)d\gamma d\xi, \label{ik5}
\end{eqnarray}
where
$\frac{dx_k(\xi,\omega_j(\xi))}{d\xi}$
can be computed using the identity
$\omega_k(x_k(\xi,\omega_j(\xi)),\xi,\omega_j(\xi)) \equiv s.$
Indeed, this and (\ref{char}) yield
\begin{eqnarray}
& & \frac{dx_k(\xi,\omega_j(\xi))}{d\xi} = -a_k( x_k(\xi,\omega_j(\xi)), \omega_k(x_k(\xi,\omega_j(\xi)),\xi,\omega_j(\xi)))  \nonumber\\
& & \times
\left(\partial_2\omega_k(x_k(\xi,\omega_j(\xi)),\xi,\omega_j(\xi)) + a_j^{-1}(\xi,\omega_j(\xi))
\partial_3\omega_k(x_k(\xi,\omega_j(\xi)),\xi,\omega_j(\xi))\right),
\label{ik53}
\end{eqnarray}
where $\partial_2 \omega_k$ and $\partial_3 \omega_k$ are given by the formulas (\ref{dx}) and (\ref{dt}), respectively.

 Using the change of variables $\xi \to \theta = x_k(\xi,\omega_j(\xi))$
with the inverse $\xi = x_{jk}(\theta,x,t),$ one rewrites (\ref{ik5}) in the form
\begin{eqnarray}
 & & \int_0^{x_k(x,t)}\int_0^1 \frac{dx_k(\xi,\omega_j(\xi))}{d\xi}d_{jki}(\xi,\theta,x,t)[\omega_k - \tilde\omega_k]
 (\theta,\xi,\omega_j(\xi))  \nonumber\\
 & & \times \rho^{-1}(\xi,\theta,x,t) v_i\Bigl(\theta,[\gamma\omega_k + (1-\gamma)\tilde\omega_k](\theta,\xi,\omega_j(\xi))\Bigl)\Bigl|_{\xi =
 x_{jk}(\theta,x,t)} \nonumber\\
 & & \times \frac{\partial}{\partial\theta}x_{jk}(\theta,x,t)d\gamma d\theta, \label{in6}
\end{eqnarray}
where the derivative $\frac{\partial}{\partial\theta}x_{jk}(\theta,x,t)$ can be easily computed from the identity (\ref{2.60})
as 
$$\frac{\partial}{\partial\theta}x_{jk}(\theta,x,t) = \frac{\partial_2 \omega_k(x_{jk}(\theta,x,t),\theta,s)[a_k a_j](x_{jk}(\theta,x,t),
\omega_j(x_{jk}(x,t),x,t))}{[a_k - a_j](x_{jk}(\theta,x,t),\omega_j(x_{jk}(x,t),x,t))}. $$
Taking into account (\ref{in6}), the expression $I_1$ given by (\ref{ik4}) now reads
\begin{eqnarray}
 & & I_1(x,t)=\int_{x_k(x,t)}^x d_{jki}(x,\eta,x,t)(\omega_k(\eta) - \tilde\omega_k(\eta)) \nonumber\\
 & & \times \int_0^1 \rho^{-1}(x,\eta,x,t,\gamma) v_i\Bigl(\eta,\gamma \omega_k(\eta) + (1-\gamma)\tilde\omega_k(\eta)\Bigl) d\gamma d\eta  \nonumber\\
 & & + \int_0^{x_k(x,t)} \int_0^1 \frac{dx_k(\xi,\omega_j(\xi))}{d\xi}x_k(\xi,\omega_j(\xi)) d_{jki}(\xi,\eta,x,t)[\omega_k - \tilde\omega_k]
 (\theta,\xi, \omega_j(\xi))  \nonumber\\
 & & \times \rho^{-1}(\xi,\theta,x,t,\gamma) \frac{\partial x_{jk}(\theta,x,t)}{\partial\theta}
 v_i\Bigl(\theta,[\gamma \omega_k + (1-\gamma)\tilde\omega_k](\theta,\xi,\omega_j(\xi))\Bigl)\Bigl|_{\xi=x_{jk}(\theta,x,t)}d\gamma d\theta \nonumber\\
 & & - \int_{x_j(x,t)}^x \int_{x_k(\xi,\omega_j(\xi))}^\xi \int_0^1 \frac{d}{d\xi} \Bigl[ d_{jki}(\xi,\eta,x,t)
 (\omega_k(\eta,\xi,\omega_j(\xi)) - \tilde\omega_k(\eta,\xi,\omega_j(\xi))) \nonumber\\
 & & \times  \rho^{-1}(\xi,\eta,x,t,\gamma)\Bigl]
 v_i\Bigl(\eta, \gamma \omega_k(\eta,\xi,\omega_j(\xi)) + (1-\gamma)\omega_k(\eta,\xi,\omega_j(\xi))\Bigl) d\gamma d\eta d\xi. \label{ik52}
\end{eqnarray}

We are prepared to derive the desired upper bound for $|I_1|.$ To this end, we use the estimates
(\ref{eq:h11}), (\ref{k3}), (\ref{al2}) and apply the Cauchy-Schwarz inequality to \reff{ik52}.
As a result,  we derive the estimate
\begin{eqnarray*}
 |I_1(x,t)| \le C_2 \tilde\beta(\varepsilon) \max_{\theta \in [s,s+3d]} \left(\int_0^1 |v_i(\eta,\theta)|^2 d\eta\right)^{1/2} \le
 K C_2 \tilde\beta(\varepsilon) e^{3d\nu} \|\varphi\|_{L^2((0,1);\R^n)},
\end{eqnarray*}
the constant $C_2$ being independent of $s$, $\vphi$, and $\eps$.

\vspace{2mm}
{\it Step 9.
Obtaining an upper bound of the type $\be(\eps)\|\vphi\|_{L^2\left((0,1);\R^n\right)}$ for the
first  summand in the right-hand side of \reff{rep_(u-v)4}.
}
\vspace{2mm}
Again, we consider the case $i=0$ and estimate $D(\tilde Q-Q) v$ (the proof of $i\ge 1$ uses a similar arguments).
Our starting point is the formula
\begin{eqnarray} \label{DQQ}
& & \left[D(\tilde Q-Q) v\right]_j(x,t)  = \nonumber\\
& & =\sum_{ k\not=j} \int_{x_j(x,t)}^x  d_j(\xi,x,t)
b_{jk}(\xi,\om_j(\xi))\left[(\tilde Q-Q) v\right]_k(\xi,\omega_j(\xi))d\xi \nonumber\\
& &
= \sum_{ k\not=j} \int_{x_j(x,t)}^x  d_j(\xi,x,t)
\Bigl(
\tilde c_k\left(\tilde x_k(\xi,\omega_j(\xi)),\xi,\omega_j(\xi)\right)-
c_k\left( x_k(\xi,\omega_j(\xi)),\xi,\omega_j(\xi)\right)
\Bigl)(\tilde Pv_k)(\xi,\omega_j(\xi))\,d\xi
\nonumber\\
& &
+\sum_{ k\not=j} \int_{x_j(x,t)}^x  d_j(\xi,x,t)b_{jk}(\xi,\om_j(\xi))c_k\left( x_k(\xi,\omega_j(\xi)),\xi,\omega_j(\xi)\right)
\left[(\tilde P-P)v_k\right](\xi,\omega_j(\xi))\,d\xi,
\end{eqnarray}
where
\begin{eqnarray} \label{ik9}
 (Pv)_k(x,t) = \left\{\begin{array}{ll}
                 (Rv)_k(w_k(x_k(x,t))), & \ x_k(x,t) \not\in (0,1), \\
                 \varphi_k(x_k(x,t)), &\ x_k(x,t) \in (0,1),
               \end{array}\right.
\end{eqnarray}
while $\tilde P$ is given by the formula (\ref{ik9}) with $\tilde\omega_k$ and $\tilde x_k$ in place of
$\omega_k$ and $x_k$, respectively. Next, we use \reff{al2}$_1$ to conclude that for all $j,k \le n$
\begin{eqnarray} \label{ik1}
 \max_{\xi,x \in [0,1]} \max_{s\le t\le s+3d}\left|\tilde c_k(\tilde x_k(\xi,\omega_j(\xi)),\xi,\omega_j(\xi)) -
c_k(x_k(\xi,\omega_j(\xi)),\xi,\omega_j(\xi))\right|\le C_3\tilde\be(\eps),
\end{eqnarray}
where $C_3$ does not depend on $s$ and $\varepsilon.$ Applying the inequality (\ref{ik1})
to the first sum in the right-hand side  of (\ref{DQQ}) and using the bound (\ref{k3}),
we estimate the absolute value of
this summand 
from above by $C_4\tilde\be(\eps)$, where the positive constant $C_4$ does not depend on $s$ 
and~$\varepsilon$.

Now we aim at estimating the second sum in (\ref{DQQ}), denoted further by 
$$
I_2(x,t)=\sum_{k\ne j}I_{2k}(x,t).
$$
To this end, fix $j,k\le m$ (for the other $j,k$ we proceed similarly), and let
 $\tilde x_{jk}(\theta,x,t)$ denote the value of $\xi$ at which the characteristics $\omega_j(\xi,x,t)$
and $\tilde\omega_k(\xi,\theta,s)$ intersect (if they do). Note that $\tilde x_{jk}(\theta,x,t)$ fulfills the equation
$$
\omega_j(\tilde x_{jk}(\theta,x,t),x,t) = \tilde\omega_k(\tilde x_{jk}(\theta,x,t),\theta,s).
$$
Suppose that $a_j(x,t) > a_k(x,t)$ (the case $a_j(x,t) < a_k(x,t)$ is treated similarly).
Then
\begin{eqnarray}\label{ik11}
 & & I_{2k} (x,t )=   \int_{x_j(x,t)}^{\min\{x_{jk}(0,x,t),\tilde x_{jk}(0,x,t)\}} d_j(\xi,x,t)
b_{jk}(\xi,\omega_j(\xi))c_k(x_k(\xi,\omega_j(\xi)),\xi,\omega_j(\xi)) \nonumber\\
 & & \times \Bigl((Rv)_k(\tilde\omega_k(\tilde x_k(\xi,\omega_j(\xi)),\xi,\omega_j(\xi))) -
(Rv)_k(\omega_k(x_k(\xi,\omega_j(\xi)),\xi,\omega_j(\xi)))\Bigl) d\xi \nonumber \\
& & +  \int_{\min\{x_{jk}(0,x,t),\tilde x_{jk}(0,x,t)\}}^{\max\{x_{jk}(0,x,t),\tilde x_{jk}(0,x,t)\}} d_j(\xi,x,t)
b_{jk}(\xi,\omega_j(\xi))c_k(x_k(\xi,\omega_j(\xi)),\xi,\omega_j(\xi))  \nonumber\\
 & & \times [(\tilde P - P)v]_k(\xi,\omega_j(\xi))d\xi  \nonumber\\
 & & +    \int_{\max\{x_{jk}(0,x,t),\tilde x_{jk}(0,x,t)\}}^x d_j(\xi,x,t)
b_{jk}(\xi,\omega_j(\xi))c_k(x_k(\xi,\omega_j(\xi)),\xi,\omega_j(\xi)) \nonumber\\
 & & \times \Bigl(\varphi_k( \tilde x_k(\xi,\omega_j(\xi))) - \varphi_k( x_k(\xi,\omega_j(\xi)))\Bigl) d\xi  =
 I_{2k1}(x,t ) + I_{2k2}(x,t ) +I_{2k3}(x,t ).
\end{eqnarray}
To estimate the second summand $I_{2k2}$, first derive the bound
\begin{eqnarray} \label{ik10}
 |x_{jk}(0,x,t) - \tilde x_{jk}(0,x,t)| \le \beta_2(\varepsilon),
\end{eqnarray}
where $\beta_2(\varepsilon)\to 0$ as $\eps\to 0$. Recall that we are in the case $j, k \le m$ and $a_j(x,t) > a_k(x,t),$ and
for the other $j,k$ we proceed similarly.

Characteristic functions $\sigma_k(\tau,0,s)$ and $\tilde\sigma_k(\tau,0,s)$ are solutions to the initial value
problems
\begin{eqnarray} \label{xjk}
 \frac{dx}{d\tau} = a_k(x,\tau), \quad x(s) = 0
\end{eqnarray}
and
\begin{eqnarray} \label{xjk1}
  \frac{dx}{d\tau} = a_k(x,\tau) + \tilde a_k(x,\tau), \quad x(s) = 0,
\end{eqnarray}
respectively. Changing the variables $(x,\tau) \to (y,\theta)$ by
 $x = y, \tau = \omega_j(y,1,\theta),$ the equations (\ref{xjk}) and (\ref{xjk1}) can be transformed as follows:
\begin{eqnarray} \label{xjk2}
 \frac{dy}{d\theta} = \frac{a_j(y,\tau) a_k(y,\tau)\partial_3 \omega_j(y,1,\theta)}{a_j(y,\tau) - a_k(y,\tau)}\Biggl|_{\tau = \omega_j(y,1,\theta)}
\end{eqnarray}
and
\begin{eqnarray} \label{xjk3}
 \frac{dy}{d\theta} = \frac{a_j(y,\tau) (a_k(y,\tau) + \tilde a_k(y,\tau))\partial_3 \omega_j(y,1,\theta)}
 {a_j(y,\tau) - a_k(y,\tau) - \tilde a_k(y,\tau)}\Biggl|_{\tau = \omega_j(y,1,\theta)},
\end{eqnarray}
respectively.
Write $\theta_0 = \omega_j(1,0,s)$ and estimate the difference of solutions $y_1(\theta)$ and $y_2(\theta)$
with the same initial values $y_1(\theta_0) = y_2(\theta_0) = 0$
to the equations (\ref{xjk2}) and  (\ref{xjk3}),
respectively. We have
\begin{eqnarray*}
  & & \frac{dy_1}{d\theta} -  \frac{dy_2}{d\theta} =
  -\frac{a_j^2(y_2,\tau_2)\tilde a_k(y_2,\tau_2)\partial_3 \omega_j(y_2,1,\tau_2)}
 {(a_j(y_2,\tau_2) - a_k(y_2,\tau_2) - \tilde a_k(y_2,\tau_2))(a_j(y_2,\tau_2) - a_k(y_2,\tau_2))}  \\
  & & +\frac{a_j(y_1,\tau_1) a_k(y_1,\tau_1)\partial_3 \omega_j(y_1,1,\theta)}{a_j(y_1,\tau_1) - a_k(y_1,\tau_1)} -
 \frac{a_j(y_2,\tau_2) a_k(y_2,\tau_2)\partial_3 \omega_j(y_2,1,\theta)}
 {a_j(y_2,\tau_2) - a_k(y_2,\tau_2)},  
\end{eqnarray*}
where $\tau_1 = \omega_j(y_1,1,\theta), \ \tau_2 = \omega_j(y_2,1,\theta).$
By (\ref{eq:h11}), $|a_j - a_k - \tilde a_k| \ge \Lambda_0$.
Using the  Gronwall's argument, we derive
\begin{eqnarray*}
 |y_1(\theta) - y_2(\theta)| \le C_5 \|\tilde a\|_{BC(\Pi;\M_n)}, \ \theta \in [s, s + 3d],
\end{eqnarray*}
where positive constant $C_5$ does not depend on $s$.
Geometrically, $y_1(\theta)$ is the abscissa of the point where characteristics $\omega_j(y,1,\theta)$ and
$\sigma_k(\tau,0,s)$ intersect. Given $(x,t)$, write
$\theta_1 = \omega_j(1,x,t).$
Then $y_1(\theta_1) = x_{jk}(0,x,t)$ and $y_2(\theta_1) = \tilde x_{jk}(0,x,t).$
This  yields the desired estimate (\ref{ik10})  for all $\|\tilde a\|_{BC(\Pi;\M_n)}  \le \varepsilon.$

Now, using the mean value theorem and the exponential estimate (\ref{k3}),
we easily get
$$|I_{2k2}(x,t )| \le C_6 \tilde\beta(\varepsilon)e^{3d\nu}\|\varphi\|_{L^2((0,1);\mathbb R^n)} \le \beta_3(\varepsilon)\|\varphi\|_{L^2((0,1);\mathbb R^n)},$$
where $C_6$ does not depend on $\eps$, $\vphi$, $k$ and $s$, while the function
$\be_3:[0,1]\to\R$
 approaches zero as $\eps\to 0$.
 
Returning to (\ref{ik11}), we proceed with the summand 
 \begin{eqnarray*}
& & |I_{2k3}(x,t )| = \Bigl|\int_{\max\{x_{jk}(0,x,t),\tilde x_{jk}(0,x,t)\}}^x d_{j}(\xi,x,t) b_{jk}(\xi,\omega_j(\xi))
c_k(x_k(\xi,\omega_j(\xi)),\xi,\omega_j(\xi)) \\
& & \times (\tilde x_k(\xi,\omega_j(\xi)) - x_k(\xi,\omega_j(\xi)))  \\
& & \times \int_0^1 \varphi'_k(\gamma \tilde x_k(\xi,\omega_j(\xi)) + (1-\gamma)x_k(\xi,\omega_j(\xi))) d\gamma d\xi\Bigl|.
\end{eqnarray*}
Using the notation (see (\ref{ik53}))
\begin{eqnarray}
& & \rho_k(\xi,x,t,\gamma) = \frac{d}{d\xi}\Bigl( \gamma \tilde x_k(\xi, \omega_j(\xi)) + (1-\gamma)x_k(\xi, \omega_j(\xi))\Bigl)  \nonumber\\[2mm]
& & = -\gamma \left[\frac{a_k + \tilde a_k - a_j}{a_j(a_k + \tilde a_k)}\right]
(\xi, \omega_j(\xi)) \nonumber\\[2mm]
& & \times \partial_3 \tilde\omega_k(\tilde x_k(\xi, \omega_j(\xi)),\xi,\omega_j(\xi))
[a_k + \tilde a_k](\tilde x_k(\xi, \omega_j(\xi)), \tilde\omega_k(\tilde x_k(\xi, \omega_j(\xi)),\xi,\omega_j(\xi)))  \nonumber\\[2mm]
& & - (1-\gamma)\left[\frac{a_k - a_j}{a_j a_k}\right]
(\xi, \omega_j(\xi)) \nonumber\\[2mm]
& & \times \partial_3 \omega_k( x_k(\xi, \omega_j(\xi)),\xi,\omega_j(\xi))
a_k(x_k(\xi, \omega_j(\xi)),\omega_k(x_k(\xi, \omega_j(\xi)),\xi,\omega_j(\xi))), \label{ik14}
\end{eqnarray}
we get
 \begin{eqnarray*}
& & \varphi'_k(\gamma \tilde x_k(\xi, \omega_j(\xi)) + (1-\gamma)x_k(\xi, \omega_j(\xi)))  \\
& & =  \rho_k^{-1}(\xi,x,t,\gamma)\frac{d}{d\xi}\varphi_k(\gamma \tilde x_k(\xi, \omega_j(\xi)) + (1-\gamma)x_k(\xi, \omega_j(\xi))).
\end{eqnarray*}
Notice that $\rho_k(\xi,x,t,\gamma) \not= 0$ for all $\xi,x \in [0,1], t \in [s,s+3d],$ $\gamma \in[0,1],$ and $k \le n.$
Hence,
\begin{eqnarray*}
& & |I_{2k3}(x,t )| \le \Bigl|\int_{\max\{x_{jk}(0,x,t),\tilde x_{jk}(0,x,t)\}}^x \frac{d}{d\xi}\Bigl[d_{j}(\xi,x,t) b_{jk}(\xi,\omega_j(\xi))
c_k(x_k(\xi,\omega_j(\xi)),\xi,\omega_j(\xi))  \\
& & \times (\tilde x_k(\xi,\omega_j(\xi)) - x_k(\xi,\omega_j(\xi)))  \\
& & \times \int_0^1 \rho_k^{-1}(\xi,x,t,\gamma)
\varphi_k(\gamma \tilde x_k(\xi,\omega_j(\xi)) + (1-\gamma)x_k(\xi,\omega_j(\xi))) d\gamma\Bigl] d\xi\Bigl|  \\
& & + \Bigl|\int_{\max\{x_{jk}(0,x,t),\tilde x_{jk}(0,x,t)\}}^{x} \frac{d}{d\xi} \Bigl[d_{j}(\xi,x,t) b_{jk}(\xi,\omega_j(\xi))
c_k(x_k(\xi,\omega_j(\xi)),\xi,\omega_j(\xi))  \\
& & \times (\tilde x_k(\xi,\omega_j(\xi)) - x_k(\xi,\omega_j(\xi))) \rho_k^{-1}(\xi,x,t,\gamma)\Bigl]  \\
& & \times \varphi_k( \gamma \tilde x_k(\xi,\omega_j(\xi)) + (1-\gamma)x_k(\xi,\omega_j(\xi)))d\gamma d\xi \Bigl| = |I_{2k31}(x,t)| + |I_{2k32}(x,t)|.
\end{eqnarray*}
Further,
\begin{eqnarray*}
 & & |I_{2k31}(x,t )| = \Bigl| \frac{1}{a_j(x,t)} b_{jk}(x,t) c_k(x_k(x,t),x,t)( \tilde x_k(x,t) - x_k(x,t))  \\
 & & \times \int_0^1 \rho_k^{-1}(x,x,t,\gamma)\varphi_k(\gamma \tilde x_k(x,t) + (1-\gamma)x_k(x,t))d\gamma  \\
 & & - d_j(y,x,t) b_{jk}(y,\omega_j(y)) c_k(x_k(y,\omega_j(y)),y,\omega_j(y))  \\
 & & \times ( \tilde x_k(y,\omega_j(y)) - x_k(y,\omega_j(y)))  \\
 & & \times \int_0^1 \rho_k^{-1}(y,x,t,\gamma) \varphi_k(\gamma \tilde x_k(y,\omega_j(y)) + (1-\gamma)x_k(y,\omega_j(y)))d\gamma\Bigl|,
\end{eqnarray*}
where $y = y(x,t) = \max\{x_{jk}(0,x,t),\tilde x_{jk}(0,x,t)\}.$ Next,
\begin{eqnarray*}
 & & \max_{s\le t \le s+3d}\|I_{2k31}(\cdot,t)\|^2_{L^2((0,1);\mathbb R^n)} \le
  2\max_{s\le t \le s+3d} \int_0^1 \Bigl[ \int_0^1  \frac{b_{jk}(x,t)}{a_j(x,t)} c_k(x_k(x,t),x,t) \\
  & & \times  (\tilde x_k(x,t) - x_k(x,t))  \rho_k^{-1}(x,x,t,\gamma) \varphi_k(\gamma \tilde x_k(x,t) + (1-\gamma)x_k(x,t)) d\gamma \Bigl]^2 dx \\
   & & + 2\max_{s\le t \le s+3d}\int_0^1\Bigl[ d_{j}(y,x,t) b_{jk}(y,\omega_j(y)) c_k(x_k(y,\omega_j(y)),y,\omega_j(y))
\\
  & & \times  (\tilde x_k(y,\omega_j(y)) - x_k(y,\omega_j(y))) \\
  & & \times \int_0^1  \rho_k^{-1}(y,x,t,\gamma) \varphi_k(\gamma \tilde x_k(y,\omega_j(y)) +
  (1-\gamma)x_k(y,\omega_j(y))) d\gamma \Bigl]^2 dx,
\end{eqnarray*}
Changing the variables
\begin{eqnarray} \label{ik12}
 x \to z = z(x,t,\gamma) = \gamma \tilde x_{k}(x,t) + (1 - \gamma)x_{k}(x,t)
\end{eqnarray}
and
\begin{eqnarray} \label{ik13}
 x \to \eta = \eta(x,t,\gamma) = \gamma \tilde x_{k}(y,\omega_{j}(y)) + (1 - \gamma)x_{k}(y,\omega_{j}(y))
\end{eqnarray}
in the first and in the second summands, respectively, we get
\begin{eqnarray}
  & & \max_{s\le t \le s+3d}\|I_{2k31}(\cdot,t)\|_{L^2((0,1);\mathbb R^n)}^2  \nonumber\\[2mm]
  & & \le 2\tilde\beta^2(\varepsilon) \max_{s\le t \le s+3d} \int_{\gamma \tilde x_{k}(0,t) +
  (1 - \gamma)x_{k}(0,t)}^{\gamma \tilde x_{k}(1,t) + (1 - \gamma)x_{k}(1,t)}
  \Biggl[ \int_0^1 \frac{1}{a_j(Z,t)} b_{jk}(Z,t)  \nonumber\\[2mm]
  & & \times c_k(x_k(Z,t),Z,t) \rho_k^{-1}(Z,Z,t,\gamma)
\varphi_k(z)d\gamma \Biggl]^2\partial_1 Z(z,t,\gamma) dz  \nonumber\\[2mm]
 & & + 2\tilde\beta^2(\varepsilon)\max_{s\le t \le s+3d} \int_{\gamma \tilde x_k(y(0,t),\omega_j(y(0,t),0,t)) + (1-\gamma) x_k(y(0,t),\omega_j(y(0,t),0,t))}
  ^{\gamma \tilde x_k(y(1,t),\omega_j(y(1,t),1,t)) + (1-\gamma) x_k(y(1,t),\omega_j(y(1,t),1,t))}  \Bigg[d_j(y(Y,t),Y,t)\nonumber\\[2mm]
 & & \times b_{jk}(y(Y,t),\omega_j(y(Y,t),Y,t)) c_k\left(x_k(y(Y,t),\omega_j(y(Y,t),Y,t)\right),y(Y,t),\omega_j(y(Y,t),Y,t)) \nonumber\\
  & & \times \int_0^1 \rho_k^{-1}(y(Y,t),Y,t,\gamma) \varphi_k(\eta)d\gamma\Biggl]^2 \partial_1Y(\eta,t,\gamma)d\eta, \label{ik131}
\end{eqnarray}
where $Z = Z(z,t,\gamma)$ and $Y = Y(\eta,t,\gamma)$ are inverses to (\ref{ik12}) and (\ref{ik13}), respectively.
Moreover, similarly to (\ref{ik14}),
\begin{eqnarray*}
 & & \partial_1 z(x,t,\gamma) = \frac{d}{dx}\left[\gamma \tilde x_k(x,t) + (1-\gamma)x_k(x,t)\right] \\[2mm]
 & & = -\gamma \partial_2\tilde \omega_k(\tilde x_k(x,t)) \left[a_k+\tilde a_k\right](\tilde x_k(x,t),\tilde\omega_k(\tilde x_k(x,t)))  \\[2mm]
 & & -(\gamma -1) \partial_2 \omega_k( x_k(x,t)) a_k(x_k(x,t),\omega_k(x_k(x,t)))
\end{eqnarray*}
and
\begin{eqnarray*}
 & & \partial_1 \eta(x,t,\gamma) = \frac{d}{dx}\Bigl(\gamma \tilde x_k(y,\omega_j(y)) + (1-\gamma)x_k(y,\omega_j(y))\Bigl)  \\[2mm]
 & & = \Biggl( -\gamma \left[\frac{a_k+\tilde a_k - a_j}{a_j(a_k+\tilde a_k)}\right]
 ( y,\omega_j(y)) \\[2mm]
 & & \times \partial_3 \tilde\omega_k( \tilde x_k(y,\omega_j(y)),y,\omega_j(y))\left[a_k+\tilde a_k\right]( \tilde x_k(y,\omega_j(y)),
 \tilde \omega_k( \tilde x_k(y,\omega_j(y)),y,\omega_j(y)))  \\[2mm]
 & & -(1-\gamma)\left[\frac{a_k- a_j}{a_j a_k}\right](y,\omega_j(y)) \\[2mm]
 & & \times \partial_3 \omega_k( x_k(y,\omega_j(y)),y,\omega_j(y))a_k(x_k(y,\omega_j(y)),
  \omega_k( x_k(y,\omega_j(y)),y,\omega_j(y)))\Biggl)\frac{\partial y(x,t)}{\partial x},
\end{eqnarray*}
where, on the account of \reff{2.60},
\begin{eqnarray*}
  \frac{\partial y(x,t)}{\partial x} =
  \left\{\begin{array}{l}
                  \displaystyle\frac{\partial x_{jk}(0,x,t)}{\partial x} \ \ if \ y = x_{jk}(0,x,t), \\[4mm]
   \displaystyle\frac{\partial \tilde x_{jk}(0,x,t)}{\partial x} \ \ if \ y = \tilde x_{jk}(0,x,t),
               \end{array}\right.
\end{eqnarray*}
\begin{eqnarray*}
 & &   \displaystyle\frac{\partial x_{jk}(0,x,t)}{\partial x} = \frac{\partial_2 \omega_j(x_{jk}(0,x,t),x,t)}
 {\partial_1 \omega_k(x_{jk}(0,x,t),0,s) - \partial_1 \omega_j(x_{jk}(0,x,t),x,t)} = \partial_2 \omega_j(x_{jk}(0,x,t),x,t)  \\[2mm]
 & & \times \displaystyle\frac{  a_j(x_{jk}(0,x,t),\omega_j(x_{jk}(0,x,t),x,t)) a_k(x_{jk}(0,x,t),\omega_k(x_{jk}(0,x,t),0,s))}
 {a_j(x_{j k}(0,x,t), \omega_j(x_{jk}(0,x,t),x,t)) -
 a_k(x_{j k}(0,x,t), \omega_k(x_{jk}(0,x,t),0,s))},
 \end{eqnarray*}
Due to (\ref{dx}) and \reff{eq:h11}, the right hand side is bounded uniformly in
$(x,t)\in\overline\Pi$ and $s\in\R$. A similar argument is applied also to $\frac{\partial \tilde x_{jk}(0,x,t)}{\partial x}$.

As it now easily follows from (\ref{ik131}),
\begin{eqnarray*}
  \max_{s\le t \le s+3d}\|I_{2k31}(\cdot,t)\|_{L^2((0,1);\mathbb R^n)} \le  \beta_4(\varepsilon)\|\varphi\|_{L^2((0,1);\mathbb R^n)},
\end{eqnarray*}
where the function $\beta_4(\varepsilon)$ approaches zero as $\eps\to 0$.

The summand $I_{2k32}$ can be treated similarly, this time using the change of variables
\begin{eqnarray*}
 \xi \to \eta(\xi,x,t,\gamma)= \gamma\tilde x_{k}(\xi,\omega_j(\xi)) + (1-\gamma)x_{k}(\xi,\omega_j(\xi)).
\end{eqnarray*}
Therewith we  complete estimation of  the summand $I_{2k3}$.

Returning to the formula (\ref{ik11}) again, we are left with the summand $I_{2k1},$ for which we will use the same argument as for
$I_{2k3}.$ Indeed, the mean value theorem yields the representation
\begin{eqnarray*}
 &&I_{2k1}(x,t) = \int_{x_j(x,t)}^{\min\{x_{jk}(0,x,t),\tilde x_{jk}(0,x,t)\}} d_i(\xi,x,t)
b_{jk}(\xi,\omega_j(\xi))c_k(x_k(\xi,\omega_j(\xi)),\xi,\omega_j(\xi)) \nonumber\\
\nonumber\\ & & \times \left[
\tilde\omega_k(\tilde x_k(\xi,\omega_j(\xi)),\xi,\omega_j(\xi))-\omega_k(x_k(\xi,\omega_j(\xi)),\xi,\omega_j(\xi))
\right]\nonumber\\& & \times
\int_0^1\frac{d}{d\tau}\left[\sum\limits_ {i=m+1}^{n}p_{ki}v_i(0,\tau(\ga,\xi))+\sum\limits_ {i=1}^{m}
p_{ki}v_i(1,\tau(\ga,\xi))\right]\,d\ga d\xi=\sum\limits_ {i=1}^{n}I_{2k1i}(x,t),
\end{eqnarray*}
where $\tau(\ga,\xi)=\ga\tilde\omega_k(\tilde x_k(\xi,\omega_j(\xi)),\xi,\omega_j(\xi))+(1-\ga)\omega_k(x_k(\xi,\omega_j(\xi)),\xi,\omega_j(\xi))$. Fix $i\le m$ (for $m+1\le i\le n$ we use the same argument) and proceed with the summand
$I_{2k1i}$. Similarly to the above, first note the identity
\begin{eqnarray*}
 &&
\frac{d}{d\xi}v_i(1,\tau(\ga,\xi))=\frac{d}{d\tau}v_i(1,\tau(\ga,\xi))\d_\xi\tau(\ga,\xi),
\end{eqnarray*}
and, hence,
\begin{eqnarray*}
 &&
\frac{d}{d\tau}v_i(1,\tau(\ga,\xi))=\left[\d_\xi\tau(\ga,\xi)\right]^{-1}\frac{d}{d\xi}v_i(1,\tau(\ga,\xi)).
\end{eqnarray*}
Substituting the latter into the summand $I_{2k1i}$ and integrating by parts, we easily arrive at the desired estimate for
this summand.

Summarizing, the final estimate for $I_{2}$ is as follows:
\begin{eqnarray*}
  \max_{s\le t \le s+3d}\|I_{2}(\cdot,t)\|_{L^2((0,1);\mathbb R^n)} \le  \beta_5(\varepsilon)\|\varphi\|_{L^2((0,1);\mathbb R^n)},
\end{eqnarray*}
where the function $\beta_5(\varepsilon)$ approaches zero as $\eps\to 0$. This means that we finish
with the upper bound  for the first summand in \reff{rep_(u-v)4}.

The proof is therewith complete.
\end{proof}

\section{Abstract setting}\label{sec:abstr}
\renewcommand{\theequation}{{\thesection}.\arabic{equation}}
\setcounter{equation}{0}

\subsection{Formulation of the abstract problem}

Let us write down  the  linear nonhomogeneous problem
(\ref{eq:vg}), (\ref{eq:3}), (\ref{eq:in})
in the form of an abstract evolution equation in   $L^2\left((0,1);\R^n\right)$.
As usually, by $H^1\left((0,1);\R^n\right)$ we denote the Sobolev space of all
functions $u\in L^2\left((0,1);\R^n\right)$ whose distributional derivative 
$u^\prime$ is in $L^2\left((0,1);\R^n\right)$.
Denote
$$
v(t)=(u_{1}(0,t),\dots u_m(0,t),u_{m+1}(1,t),\dots u_n(1,t))
$$
and   define a one-parameter family of operators
$\A(t)$  from $L^2\left((0,1);\R^n\right)$
to $L^2\left((0,1);\R^n\right)$ for each $t\in \R$
by
$$(\A(t)u)(x)=\left(-a(x,t)\frac{\d}{\d x} - b(x,t)\right)u,$$
with the domain 
$$
\begin{array}{ll}
\displaystyle
 D(\A(t))=\left\{u\in H^1\left((0,1);\R^n\right)\,:\,\,v(t)=
(Ru)(t)\right\}\subset L^2\left((0,1);\R^n\right),
\end{array}
$$
where the operator   $R$  is given by  (\ref{eq:R}). Note that $D(\A(t))=D$
is independent of $t$.

Writing  $u(t)$ and  $f(t)$, we mean  bounded and continuous maps $u:\R\to L^2((0,1);\R^n)$  and
$f:\R\to L^2((0,1);\R^n)$ defined by
$[u(t)](x)=u(x,t)$ and $[f(t)](x)=f(x,t)$, respectively.
In this notation, the  problem (\ref{eq:vg}), (\ref{eq:3}), (\ref{eq:in})  can be written in the abstract form
\beq\label{unperturb}
\frac{d}{dt}u=\A(t)u +f(t), \quad u(s)=\varphi\in L^2((0,1);\R^n).
\ee

Given $\vphi\in D,$  a function $u\in C\left([s,\infty);L^2((0,1);\R^n)\right)$ is called a {\it classical solution to
the abstract problem \reff{unperturb}} if
$u$ is continuously differentiable in $L^2((0,1);\R^n)$ for $t>s$, $u(t)\in D$ for $t>s$ and
\reff{unperturb}  is satisfied in $L^2((0,1);\R^n)$.

\subsection{
Equivalence between the original and the abstract problem settings
}\label{sec:equiv}

Here we show that, if $\vphi\in D$, then the
$L^2$-generalized solution to the problem
(\ref{eq:vg}), (\ref{eq:3}), (\ref{eq:in}) is a classical solution to the abstract problem \reff{unperturb} and vice versa. 

\begin{thm}\label{distr-clas}
Suppose that $a, b \in BC^1(\overline\Pi;\M_n)$,   $f \in BC^1(\overline\Pi;\R_n)$, and  
 the condition \reff{eq:h11}  is fulfilled.
If  $\vphi\in D$ and $u(x,t)$ is the 
$L^2$-generalized solution to the problem (\ref{eq:vg}), (\ref{eq:3}), (\ref{eq:in}), then 
  the function
$u(t)$ such that $[u(t)](x):=u(x,t)$,
is a classical solution  to the abstract problem \reff{unperturb}.
Vice versa, if $u(t)$ is a classical solution  to the abstract problem \reff{unperturb},
then $u(x,t):=[u(t)](x)$ is an $L^2$-generalized solution to the problem
(\ref{eq:vg}), (\ref{eq:3}), (\ref{eq:in}).
\end{thm}

The proof of the theorem is based on 
Lemmas \ref{H1}--\ref{cont-distr} below.

\begin{lemma}\label{H1}
Let the initial function $\varphi$ belongs to  $C^1([0,1];\R^n)$ and fulfills the zero order
compatibility conditions \reff{eq:nl1}.
Then there exist constants $K_2$ and $\nu_2$ such that
   the  piecewise continuously differentiable
 solution  $u$ to the problem  (\ref{eq:vg}), (\ref{eq:3}), (\ref{eq:in}) (ensured by Theorem~\ref{km}~$(\io)$)
fulfills the  estimate
\beq\label{eq:apr4}
\begin{array}{cc}
\|u(\cdot,t)\|_{H^1\left((0,1);\R^n\right)}+\|\d_tu(\cdot,t)\|_{L^2\left((0,1);\R^n\right)}\\
\le  K_2 e^{\nu_2(t-s)}\left(\|\vphi\|_{H^1((0,1);\R^n)}+
\sup\limits_{t\in\R}\|f(\cdot,t)\|_{L^2\left((0,1);\R^n\right)}+
\sup\limits_{t\in\R}\|\d_tf(\cdot,t)\|_{L^2\left((0,1);\R^n\right)}\right)
\end{array}
\ee
for all  $t\ge s$.
\end{lemma}

\begin{proof} We proceed similarly to  \cite[Lemma 4.2]{KmLyul}.
Take a scalar product  of (\ref{eq:vg}) and $u$ in $\R^n$ and integrate the resulting system over the domain
$\Pi_s^t$. We get
$$
\int\int_{\Pi_s^t}\left(\frac{\partial }{\partial \theta}(u,u)+\frac{\partial }{\partial x}(au,u)\right)\,\, dx d\theta=
\int\int_{\Pi_s^t}\left(-2(bu,u)+(\partial_xa\,u,u)+2(f,u)\right)\, dx d\theta.
$$
Here and in what follows,  $(\cdot,\cdot)$ denotes
the scalar
product
 in $\R^n$.
  Applying  Green's formula to the left hand side, we obtain
\beq\label{god04}
\begin{array}{cc}
\displaystyle\|u(\cdot,t)\|_{L^2\left((0,1);\R^n\right)}^2+\int_s^t \left(\sum_{j=1}^n a_j(1,\theta)
 u_j^2(1,\theta)-\sum_{j=1}^n
a_j(0,\theta) u_j^2(0,\theta)\right)\,d\theta\\
\displaystyle=\|\vphi\|^2_{L^2\left((0,1);\R^n\right)}+\int\int_{\Pi_s^t}\left(-2(bu,u)+(\d_xa u,u)+2(f,u)\right)\,
dxd\theta.
\end{array}
\ee

Suppose first that the boundary  
conditions  (\ref{eq:3}) are dissipative, i.e.
\beq\label{god1}
\sum\limits_{j=1}^ma_j(1,t)u_j^2(1,t)-\sum\limits_{j=m+1}^na_j(0,t)u_j^2(0,t)+
\sum\limits_{j=m+1}^na_j(1,t)(Ru)_j^2(t)-\sum\limits_{j=1}^ma_j(0,t)(Ru)_j^2(t)\ge 0.
\ee
Then from  \reff{god04} we have
\beq\label{mu1}
\begin{array}{lr}
\displaystyle\|u(\cdot,t)\|_{L^2\left((0,1);\R^n\right)}^2\le
\displaystyle\|\vphi\|^2_{L^2\left((0,1);\R^n\right)}+\int\int_{\Pi_s^t}\left|\left((\d_xa-2b)u,u\right)+2(f,u)\right|\,dx d\theta\\
\le \displaystyle\|\vphi\|^2_{L^2\left((0,1);\R^n\right)}+(t-s)\sup\limits_{t\in\R}\|f(\cdot,t)\|^2_{L^2\left((0,1);\R^n\right)}+
\kappa_1\int_s^t \|u(\cdot,\theta)\|_{L^2\left((0,1);\R^n\right)}^2\, d\theta,
\end{array}
\ee
where
$\kappa_1=n \|\d_xa-2b\|_{BC(\Pi;\M_n)}+1$.

Let us show that the inequality \reff{god1}, 
supposed above, causes no loss of generality.
Let  $\mu_j(x,t)$ be arbitrary smooth functions
satisfying the conditions
$$
\inf\limits_{\overline\Pi_\tau}|\mu_j|>0, \quad \sup\limits_{\overline\Pi_\tau}|\mu_j|<\infty
\quad\mbox{ for all } j\le n.
$$
The change of each variable 
$u_j$ to $v_j=\mu_ju_j$
brings the  system (\ref{eq:1u})  to
\begin{equation}\label{god2}
\partial_tv_j  + a_j(x,t)\partial_xv_j - \frac{\d_t\mu_j+a_j(x,t)\d_x\mu_j}{\mu_j}v_j + 
\sum\limits_{k=1}^nb_{jk}\frac{\mu_j}{\mu_k}v_k = 0
\end{equation}
and the boundary conditions  (\ref{eq:3})  to
\beq\label{god3}
\begin{array}{l}
\displaystyle
v_j(0,t) = 
\sum\limits_{k=1}^mp_{jk}\frac{\mu_j(0,t)}{\mu_k(1,t)}v_k(1,t)+
\sum\limits_{k=m+1}^np_{jk}\frac{\mu_j(0,t)}{\mu_k(0,t)}v_k(0,t),
\quad 1\le j\le m,
\nonumber\\
\displaystyle
v_j(1,t) =\sum\limits_{k=1}^mp_{jk}\frac{\mu_j(1,t)}{\mu_k(1,t)}v_k(1,t)+
 \sum\limits_{k=m+1}^np_{jk}\frac{\mu_j(1,t)}{\mu_k(0,t)}v_k(0,t),
\quad   m< j\le n.
\nonumber
\end{array}
\ee
Note that the  resulting system
 \reff{god2}, \reff{god3} is of the type  (\ref{eq:1u}),  (\ref{eq:3}), and  the  
inequality \reff{god1} for it reads
\beq\label{god4}
\begin{array}{ll}
\displaystyle
\sum\limits_{j=1}^ma_j(1,t)v_j^2(1,t)-\sum\limits_{j=m+1}^na_j(0,t)v_j^2(0,t)\\
+\displaystyle\sum\limits_{j=m+1}^na_j(1,t)
\left[
\sum\limits_{k=m+1}^np_{jk}\frac{\mu_j(1,t)}{\mu_k(0,t)}v_k(0,t)+
\sum\limits_{k=1}^mp_{jk}\frac{\mu_j(1,t)}{\mu_k(1,t)}v_k(1,t)
\right]^2\\
\displaystyle
-\displaystyle\sum\limits_{j=1}^ma_j(0,t)
\left[
\sum\limits_{k=m+1}^np_{jk}\frac{\mu_j(0,t)}{\mu_k(0,t)}v_k(0,t)+
\sum\limits_{k=1}^mp_{jk}\frac{\mu_j(0,t)}{\mu_k(1,t)}v_k(1,t)
\right]^2
\ge 0.
\end{array}
\ee 
One can easily see that the functions  $\mu_j$ can be chosen so that  the left hand side of \reff{god4}  
is a non-negative definite quadratic form with respect to
 $v_j(1,t),$ $j\leq m$ and $v_j(0,t),$ $m+1\leq j\leq n$. This finishes the proof of the desired statement.

Further we will estimate $\|\d_tu(\cdot,t)\|_{L^2\left((0,1);\R^n\right)}$. With this aim,
set
$$
v=\d_tu,
$$
where  $\d_t$ denotes the distributional derivative. Formal differentiation
of (\ref{eq:vg}) and   (\ref{eq:3}) in  $t$ (in a distributional sense)
combined with (\ref{eq:vg}) gives
\begin{equation}\label{eq:112}
\partial_tv + a\partial_xv +
\left(b-a^{-1}\d_ta\right)v
+ \left(\d_tb-a^{-1}\d_ta \,b\right) u= \d_tf - a^{-1}\d_ta \,f 
\end{equation}
and
\begin{equation}\label{eq:115}
\begin{array}{rcl}
\displaystyle
v_j(0,t) &=&\displaystyle\sum\limits_{k=m+1}^np_{jk}v_k(0,t)+\sum\limits_{k=1}^m p_{jk}v_k(1,t)
\quad   1\le j\le m,
\\ [3mm]
\displaystyle
v_j(1,t)&=&\displaystyle
\sum\limits_{k=1}^mp_{jk}v_k(0,t)+ \sum\limits_{k=m+1}^np_{jk}v_k(1,t)
\quad   m< j\le n,
\end{array}
\end{equation}
all the equalities being understood in the distributional sense.
We endow the system \reff{eq:112}--\reff{eq:115} with
 initial conditions
\beq\label{kk1}
v(x,s)= -a(x,s)\vphi^\prime(x)-b(x,s)\vphi(x)+f(x,s).
\ee
Note that  \reff{eq:112}--\reff{kk1} is the initial-boundary value problem  with respect to $v$.

Fix an arbitrary $t\ge s$. As it follows from
Theorem \ref{km},   the vector-function $v$
is  piecewise continuous in $\overline\Pi_s^t$, with  a finite number of first order discontinuities (if any)
 along certain characteristic curves. The union of those characteristic curves
 will be denoted  by $J$. From the equation (\ref{eq:112}) we conclude that
the  generalized  directional derivatives
$$
z_j=\partial_tv_j + a_j\partial_xv_j
$$
are  continuous functions on
$\overline\Pi_s^t\setminus J$, with possible first order discontinuities on $J$.
This means  that the system  (\ref{eq:112}) is satisfied
pointwise  everywhere on  $\overline\Pi_s^t\setminus J$,
while the system (\ref{eq:115})  is satisfied everywhere on $[s,t]$ excepting a finite number of
points.

Consequently,  we have the following pointwise identity on $\Pi_s\setminus J$:
\begin{equation}\label{k2}
z + \left(b-a^{-1}\d_ta \right)v
+ \left(\d_tb-a^{-1}\d_ta \,b\right) u = \d_tf - a^{-1}\d_ta \,f .
\end{equation}

Multiplying \reff{k2} by $v$ and integrating the resulting system
over the domain
$\Pi_s^t$, we get
\begin{eqnarray}
& & \int\int_{\Pi_s^t}(z,v)\, dx d\theta 
=-\int\int_{\Pi_s^t}\left(\left(b-a^{-1}\d_ta \right)v
+ \left(\d_tb-a^{-1}\d_ta \,b \right) u,v\right)\, dx d\theta \nonumber\\
& & + \int\int_{\Pi_s^t}(\d_tf - a^{-1}\d_ta \,f,v)\, dx d\theta. \label{k33}
\end{eqnarray}

Since $C^1\left(\overline\Pi_s^t;\R^n\right)$ is densely embedded
into $L^2\left(\Pi_s^t;\R^n\right)$,
there is a   sequence
$v^l\in C^1(\overline\Pi_s^t;\R^n)$, $l\in\N$,
such that
\beq\label{k5}
\begin{array}{ll}
v^l\to v \mbox{  in  } L^2\left(\Pi_s^t;\R^n\right) \mbox{ as } l\to\infty
\end{array}
\ee
Let us show that
\beq\label{k51}
\begin{array}{ll}
\langle\d_tv^l+a\d_xv^l,\vphi\rangle_{L^2} \to \langle z,\vphi\rangle_{L^2} \mbox{  for all  }  
\vphi\in L^2\left(\Pi_s^t;\R^n\right),
\end{array}
\ee
where 
$\langle\cdot,\cdot\rangle_{L^2} : L^2\left(\Pi_s^t;\R^n\right)\times L^2\left(\Pi_s^t;\R^n\right)
\to \R$ denotes a scalar product in $L^2\left(\Pi_s^t;\R^n\right)$.
Indeed, due to \reff{k5}, for any $\vphi\in C_0^\infty(\overline\Pi_s^t;\R^n)$ we have
$$
\begin{array}{ll}
\langle\d_tv^l+a\d_xv^l,\vphi\rangle_{L^2} = 
-\langle v^l,\d_t\vphi+\d_x(a\vphi)\rangle_\D \\ [2mm]
\to -\langle v,\d_t\vphi+\d_x(a\vphi)\rangle_\D
=\langle\d_tv+a\d_xv,\vphi\rangle_\D=\langle z,\vphi\rangle_{L^2},
\end{array}
$$
where  
$\langle\cdot,\cdot\rangle_{\D} : \D^\prime\left(\Pi_s^t;\R^n\right)\times \D\left(\Pi_s^t;\R^n\right)
\to \R$ denotes a dual pairing in $\D^\prime$ and 
$\d_t$ and $\d_x$ are understood in a distributional sense.
As the space $C_0^\infty(\overline\Pi_s^t;\R^n)$ is dense in $L^2\left(\Pi_s^t;\R^n\right)$,
the desired assertion \reff{k51}  follows.

On the account of \reff{k51}, it holds
\begin{eqnarray*}
& & \int\int_{\Pi_s^t}(z,v)\, dx d\theta\\ 
& & =
\lim\limits_{r\to\infty}
\int\int_{\Pi_s^t}(\d_tv^r+a\d_xv^r,v)\, dx d\theta=
\lim\limits_{r\to\infty}
\lim\limits_{l\to\infty}\int\int_{\Pi_s^t}(\d_tv^r+a\d_xv^r,v^l)\, dx d\theta
\\
& & =
-\lim\limits_{r\to\infty}\lim\limits_{l\to\infty}\int\int_{\Pi_s^t}
\left(v^r,\d_tv^l+\d_x(av^l)\right)\, dx d\theta
+\lim\limits_{r\to\infty}\lim\limits_{l\to\infty}\int_0^1(v^r,v^l)|_{\theta=s}^t\, dx \\
& & +
\lim\limits_{r\to\infty}\lim\limits_{l\to\infty}\int_s^t \left(\sum_{j=1}^n a_j(1,\theta)
 v_j^r(1,\theta)v_j^l(1,\theta)-\sum_{j=1}^n
a_j(0,\theta) v_j^r(0,\theta)v_j^l(0,\theta)\right)\,d\theta\\
& & =
-\lim\limits_{l\to\infty}\int\int_{\Pi_s^t}
\left(v,\d_tv^l+a\d_xv^l\right)\, dx d\theta-\int\int_{\Pi_s^t}
\left(v,\d_xa\,v\right)\, dx d\theta\\
& &  +\int_0^1\sum_{j=1}^n\left[v_j^2(x,t)-v_j^2(x,s)\right]\, dx +
\int_s^t \left(\sum_{j=1}^n a_j(1,\theta)
 v_j^2(1,\theta)-\sum_{j=1}^n
a_j(0,\theta) v_j^2(0,\theta)\right)\,d\theta.
\end{eqnarray*}
Consequently,
\begin{eqnarray}\label{k9}
& & 2\int\int_{\Pi_s^t}(z,v)\, dx d\theta
=
-\int\int_{\Pi_s^t}
\left(v,\d_xa\,v\right)\, dx d\theta +\int_0^1\sum_{j=1}^n\left[v_j^2(x,t)-v_j^2(x,s)\right]\, dx \nonumber\\ [1mm]
& &  + \int_s^t \left(\sum_{j=1}^n a_j(1,\theta)
 v_j^2(1,\theta)-\sum_{j=1}^n
a_j(0,\theta) v_j^2(0,\theta)\right)\,d\theta.
\end{eqnarray}
Combining \reff{k9} with \reff{k33}, we have
\begin{eqnarray}\label{kk2}
& & \|v(\cdot,t)\|_{L^2((0,1);\R^n)}^2+\int_s^t \left(\sum_{j=1}^n a_j(1,\theta)
 v_j^2(1,\theta)-\sum_{j=1}^n
a_j(0,\theta) v_j^2(0,\theta)\right)\,d\theta \nonumber\\ [1mm]
& & =\left\|a(\cdot,s)\vphi^\prime+b(\cdot,s)\vphi-f(\cdot,s)\right\|^2_{L^2((0,1);\R^n)}
+\int\int_{\Pi_s^t}\left(\left[\d_xa-2b+2 a^{-1}\d_ta\right]v,v\right)\,dxd\theta \nonumber\\ [1mm]
& & -2\int\int_{\Pi_s^t}\left(\left[\d_tb-a^{-1}\d_ta\,b\right]u,v\right)\,dxd\theta
+2\int\int_{\Pi_s^t}(\d_tf - a^{-1}\d_ta \, f,v)\, dx d\theta.
\end{eqnarray}
We now use the disipativity condition \reff{god1} (similarly to the above, this causes no loss of generality).
The equation  \reff{kk2} yields
\begin{eqnarray}\label{eq:101}
& & \|v(\cdot,t)\|_{L^2((0,1);\R^n)}^2\, d\theta\le
\left\|a(\cdot,s)\vphi^\prime+b(\cdot,s)\vphi-f(\cdot,s)\right\|^2_{L^2((0,1);\R^n)} \nonumber\\[1mm]
& & + \kappa_2 \int_s^t\|f(\cdot,\theta)\|_{L^2((0,1);\R^n)}^2\, d\theta
+ \int_s^t\|\d_tf(\cdot,\theta)\|_{L^2((0,1);\R^n)}^2\, d\theta \nonumber\\[1mm]
& & +
\kappa_3 \left(\int_s^t\|u(\cdot,\theta)\|_{L^2((0,1);\R^n)}^2\, d\theta+\int_s^t\|v(\cdot,\theta)\|_{L^2((0,1);\R^n)}^2\, d\theta\right),
\end{eqnarray}
where the constants  $\kappa_2$ and $\kappa_3$ depend on $a$ and $b$ but not on $f$ and $\vphi$.

Furthermore, we sum up \reff{mu1} and \reff{eq:101}. After  applying the Gronwall's argument to the resulting
inequality, we get the bound
\begin{eqnarray*}
& & \|u(\cdot,t)\|_{L^2\left((0,1);\R^n\right)}+\|\d_tu(\cdot,t)\|_{L^2\left((0,1);\R^n\right)}
 \le K_{22}e^{\nu_{22}(t-s)} \Bigl(\|\vphi\|_{H^1((0,1);\R^n)}   \\
& & \quad + \sup\limits_{t\in\R}\|f(\cdot,t)\|_{L^2\left((0,1);\R^n\right)}+
\sup\limits_{t\in\R}\|\d_tf(\cdot,t)\|_{L^2\left((0,1);\R^n\right)}\Bigl)
\end{eqnarray*}
 for all  $t\ge s$ and  some positive constants $K_{22}$ and $\nu_{22}$.

A similar estimate for
$\|\d_xu(\cdot,t)\|_{L^2\left((0,1);\R^n\right)}$  easily follows from \reff{eq:vg}. This completes
 the proof of \reff{eq:apr4}.   
\end{proof}

\begin{lemma}\label{cont_reg}
Let  $\varphi\in D$. Then

$(\io)$
the continuous solution
$u$ to the problem  (\ref{eq:vg}), (\ref{eq:3}), (\ref{eq:in})
belongs to  $C([s,t],H^1\left((0,1);\R^n\right))$ and to $C^1([s,t],L^2\left((0,1);\R^n\right))$;

$(\io\io)$  the function $U(t,s)\vphi$ for $t\ge s$ is  continuously differentiable in $t$ and satisfies
the homogeneous abstract equation \reff{unperturb} in $L^2\left((0,1);\R^n\right)$.
\end{lemma}
\begin{proof}
Define $\tilde\varphi(x) = \varphi(0) + x(\varphi(1) - \varphi(0))$ for $x \in [0,1].$
Note that $\varphi - \tilde\varphi \in H^1_0((0,1);\mathbb R^n)$, see \cite[p. 259]{Ev}. Therefore,
there exists a sequence
$\vphi^l_0\in C_0^\infty([0,1];\R^n)$ approaching $\varphi - \tilde\varphi$ in $H^1((0,1);\R^n)$.
It follows that the sequence $\varphi^l = \varphi^l_0 + \tilde\varphi$  approaches $\varphi$ in $H^1((0,1);\R^n)$.

By Theorem \ref{km} $(\io\io)$ and Lemma \ref{H1},
the piecewise continuously differentiable
 solution $u^l$ to the problem  (\ref{eq:vg}), (\ref{eq:3}), (\ref{eq:in}) with $\vphi^l$
in place of $\vphi$ satisfies the estimate  \reff{eq:apr4} with $u=u^l$ and $\vphi=\vphi^l$.
This entails the convergence
\beq\label{1k}
\max\limits_{s\le\tau\le t}\|u^m(\cdot,\tau)-u^l(\cdot,\tau)\|_{H^1\left((0,1);\R^n\right)}+
\max\limits_{s\le\tau\le t}\|\d_tu^m(\cdot,\tau)-\d_tu^l(\cdot,\tau)\|_{L^2\left((0,1);\R^n\right)}\to 0
\ee
as $m,l\to\infty$ for each $t>s$. 
Consequently, the sequence $u^l$ converges in
$C([s,t],H^1\left((0,1);\R^n\right))\cap C^1([s,t],L^2\left((0,1);\R^n\right))$. This proves Claim
$(\io)$.

Claim $(\io\io)$ now easily follows from  Claim $(\io)$.
\end{proof}

\begin{lemma}\label{L2-cont}
Let $\varphi\in D$. A function $u$ is
the $L^2$-generalized solution
to the problem  (\ref{eq:vg}), (\ref{eq:3}), (\ref{eq:in}) (see Theorem \ref{evol_g})
if and only if it is the
continuous solution
to the problem (\ref{eq:vg}), (\ref{eq:3}), (\ref{eq:in})
(see Theorem \ref{eq:nl1} $(\io)$).
\end{lemma}

\begin{proof}
{\it Necessity.} Notice first that similarly to the proof of Lemma \ref{cont_reg} there is a sequence 
$\vphi^l\in C^1([0,1];\R^n)$ approaching $\varphi$ in $H^1\left((0,1);\R^n\right))$. We 
 use the convergence \reff{1k}
for the piecewise continuously differentiable
 solution $u^l$ to the problem  (\ref{eq:vg}), (\ref{eq:3}), (\ref{eq:in}) with $\vphi^l$
in place of $\vphi$. It follows that, given $s<t$, $u^l$ converges in
$C\left([0,1]\times[s,t];\R^n\right)$. This means that the $L^2$-generalized solution
has, in fact, better regularity.

Now, since $u^l$ for each $l\in\N$ is a continuous solution, it satisfies
the system
$$
u_j^l(x,t)=(Qu^l)_j(x,t)
-\int_{x_j(x,t)}^x d_j(\xi,x,t)
\left(\sum_{k\not=j}^n b_{jk}(\xi,\om_j(\xi))
u_k^l(\xi,\om_j(\xi))-f(\xi,\om_j(\xi))\right)\,d\xi
$$
for all $j\le n$, where the operator $Q$ is given by \reff{Q}.
 Letting $l\to\infty$
finishes the proof of the necessity.

{\it Sufficiency.} We use the fact that the problem (\ref{eq:vg}), (\ref{eq:3}), (\ref{eq:in}),
according to Theorem \ref{evol_g}, has a unique $L^2$-generalized solution $\tilde u$.
By uniqueness, $\tilde u$ is the limit of $u^l$ in the sense of
Definition \ref{L2_g}. Due to \reff{1k}, $\tilde u$ is a continuous function that coincides with~$u$.
\end{proof}

\begin{lemma}\label{cont-distr}
Let $\varphi\in D$. A continuous function $u$ is
the continuous solution
to the problem  (\ref{eq:vg}), (\ref{eq:3}), (\ref{eq:in}) if and only if
$u$ satisfies (\ref{eq:vg}) in a distributional sense and (\ref{eq:3}) and (\ref{eq:in}) pointwise.
\end{lemma}

\begin{proof}
{\it Necessity.} Let $u$ be the continuous solution
to the problem  (\ref{eq:vg}), (\ref{eq:3}), (\ref{eq:in}).
It is straightforward to check that $u$ fulfills (\ref{eq:3}) and (\ref{eq:in}).
It remains to  show that $u$ satisfies (\ref{eq:vg}) in a distributional sense.
 Fix arbitrary $j\le n$ and  $s<t$ and take an arbitrary sequence $u^l\in C^{1}\left([0,1]\times[s,t];\R^n\right)$ approaching
$u$ in $C\left([0,1]\times[s,t];\R^n\right)$. Then for any smooth function
$\phi: (0,1)\times(s,t)\to\R$
with compact support we have
\begin{eqnarray*}
& & \langle (\d_t+a_j\d_x)u_j,\phi\rangle = \langle u_j,-\d_t\phi-\d_x(a_j\phi)\rangle
= \lim_{l\to\infty}\langle u_j^l,-\d_t\phi-\d_x(a_j\phi)\rangle =
\lim_{l\to\infty}\Bigl\langle 
(Qu^l)_j(x,t) \\
& & -\int_{x_j(x,t)}^x d_j(\xi,x,t)
\sum_{k\not=j}^n \left( b_{jk}(\xi,\om_j(\xi))
u_k^l(\xi,\om_j(\xi)) - f_j(\xi, \omega_j(\xi))\right) d\xi,
-\d_t\phi-\d_x(a_j\phi)\Bigl\rangle \\ 
& & = \lim_{l\to\infty}\Bigl\langle -\sum_{k=1}^n b_{jk}(x,t)u_k^l + f_j(x,t),\phi\Bigl\rangle
 =\Bigl\langle -\sum_{k=1}^n b_{jk}(x,t)u_k + f_j(x,t),\phi\Bigl\rangle
\end{eqnarray*}
as desired.
Here we used the formula
\beq\label{2k}
(\d_t+a_j(x,t)\d_x)\psi(\om_j(\xi,x,t))=0
\ee
being true for all $j\le n$, $\xi,x\in[0,1]$, $t\ge s$, and for any $\psi\in C^1(\R)$.

{\it Sufficiency.} Assume that a continuous vector-function $u$
satisfies (\ref{eq:vg}) in a distributional sense and (\ref{eq:3}) and (\ref{eq:in}) pointwise. Note the identity
\beq\label{10k}
(\d_t+a_j(x,t)\d_x)x_j(x,t)=0.
\ee
In the domain $\{(x,t)\in\Pi_s\,:\,t>\om_j(x_j(x,t),x,t)\}$ it is obvious.
In the domain
$\{(x,t)\in\Pi_s\,:\,t<\om_j(x_j(x,t),x,t)\}$ this identity easily follows from  the identity
$\om_j(x_j(x,t),x,t)=s$, after applying the operator $\d_t+a_j(x,t)\d_x$
to both sides and using the equation \reff{2k}.

On the account of \reff{2k} and \reff{10k},
we rewrite the system (\ref{eq:vg}) in the form
\begin{equation}\label{eq:vgc}
(\partial_t  + a_j(x,t)\partial_x)\left(c_j^{-1}(x_j(x,t),x,t) u_j\right)=
c_j^{-1}(x_j(x,t),x,t)
\left(-\sum\limits_{k\ne j}b_{jk}(x,t) u_k + f_j(x,t)\right),
\end{equation}
without destroying the equalities in the sense of distributions.
To prove that $u$ satisfies \reff{rep}
pointwise, we use  the constancy theorem of distribution theory claiming
that any distribution on an open set with zero generalized derivatives
is a constant on any connected component of the set.
Hence, due to \reff{eq:vgc}, for each
$j\le n$ the expression
\begin{eqnarray}\label{11k}
c_j^{-1}(x_j(x,t),x,t)\left(u_j(x,t)+\int_{x_j(x,t)}^x d_j(\xi,x,t)\Bigl[\sum\limits_{k\ne j}b_{jk}u_k - f_j\Bigl](\xi,\om_j(\xi))\,d\xi\right)
\end{eqnarray}
is a constant along the characteristic curve $\om_j(\xi,x,t)$. In other words, the distributional directional derivative
$(\partial_t  + a_j(x,t)\partial_x)$ of the function \reff{11k} is equal to zero. Since \reff{11k} is a  continuous function,
$c_j(x_j(x,t),x_j(x,t),t)=1$,
 and the trace
$u_j(x_j(x,t),t)$ is given by means of (\ref{eq:3}) and (\ref{eq:in}), it follows that
 $u$ satisfies the system \reff{rep} pointwise, as desired.
 \end{proof}

\begin{proofthm}{\it 4.1.} 
Given $\vphi\in D$,
let $u$ be the $L^2$-generalized solution 
to  the problem (\ref{eq:vg}), (\ref{eq:3}), (\ref{eq:in}).
Due to Lemmas \ref{L2-cont} and \ref{cont-distr}, this solution
satisfies (\ref{eq:vg}) in a distributional sense and (\ref{eq:3}) and (\ref{eq:in}) pointwise. 
By Lemma \ref{cont_reg}, the  distributional derivatives $\d_x u$  and $\d_t u$ 
belong in fact to $C([s,t],L^2\left((0,1);\R^n\right))$.
Consequently, $u(t)$ is a classical solution to the abstract problem \reff{unperturb}.

The converse follows from the uniqueness of the classical solution to the abstract problem~\reff{unperturb}.
\end{proofthm}

\section{Proof of the main Theorem \ref{main}}\label{sec:main}

\subsection{Bounded Solutions}
\label{sec:bounded}

Here we prove Theorem \ref{main} $(\io)$.
Suppose that the unperturbed linear system (\ref{eq:1u}), (\ref{eq:3})
is exponentially dichotomous with an exponent $\alpha > 0$, a
bound $M \ge 1$, and with the dichotomy projectors $P(t)$, $t\in\R$.
By Theorem~\ref{robust}, there exist $\varepsilon_0 > 0$,  $\alpha_1 \le \alpha$, and  $M_1 \ge M$
such that for all $\tilde a$ and $\tilde b$ with
$\|\tilde a\|_{BC^1(\Pi;\M_n)} \le \varepsilon_0$
and $\|\tilde b\|_{BC(\Pi;\M_n)} \le \varepsilon_0$ the perturbed system (\ref{eq:1p}), (\ref{eq:3}) is exponentially
dichotomous with the exponent $\alpha_1 $ and the bound $M_1 $.
Since the functions $A$ and $B$ are $C^2$-smooth,
there exists positive $\delta\le\de_0$ such that
\begin{eqnarray*}
&&\sup\left\{\left|A_j(x,t,\varphi(x,t)) - A_j(x,t,0)\right|\,:\, (x,t)\in \overline\Pi
\right\}  \\ [2mm]
&&
\quad+\sup\left\{\left|\frac{d}{dx}\left[A_j(x,t,\varphi(x,t)) - A_j(x,t,0)\right]\right|\,:\, (x,t)\in \overline\Pi
\right\}  \\ [2mm]
&&
\quad+\sup\left\{\left|\frac{d}{dt}\left[A_j(x,t,\varphi(x,t)) - A_j(x,t,0)\right]\right|\,:\, (x,t)\in \overline\Pi
\right\}\le \varepsilon_0,  \\ [2mm]
&&\sup\left\{|B_{jk}(x,t,\vphi(x,t)) - B_{jk}(x,t,0)|\,:\, (x,t)\in \overline\Pi\right\} \le \varepsilon_0.
\end{eqnarray*}
for all $\varphi \in BC^1(\overline \Pi, \mathbb R^n)$ with $\|\varphi\|_{BC^1(\Pi, \mathbb R^n)} \le \delta $
and $ 1\le j,k\le n$.
Then, given $\vphi\in BC^1(\overline\Pi;\R^n)$, the system
$$\partial_t u + A(x,t,\varphi)\partial_x u + B(x,t, \varphi)u = 0$$
with boundary conditions (\ref{eq:3}) has the exponential dichotomy with the constants $\alpha_1$ and $M_1$
whenever $\|\varphi\|_{BC^1(\Pi;\R^n)} \le \delta.$

The proof will be based on the following iteration procedure.
Put $u^0(x,t) \equiv 0$. We will obtain the iteration  $u^{k+1}(x,t)$ as the unique
$BC^2(\overline\Pi;\R^n)$-smooth bounded
solution  to the linear system
\begin{eqnarray} \label{th-2}
 \partial_t u + a^k(x,t)\partial_x u  + b^k(x,t)u = f(x,t),\quad k=0,1,2,\dots,
\end{eqnarray}
with the boundary conditions (\ref{eq:3}). Here $a^k(x,t) = A(x,t,u^k(x,t))$ \ and \ $b^k(x,t) = \\ B(x,t,u^k(x,t))$.

We divide the proof into three claims.

{\it Claim 1. 
Suppose that
\begin{eqnarray} \label{sm7aa}
\|f\|_{BC^2(\Pi;\R^n)} \le \frac{1}{L}\left(\frac{2M_1}{\alpha_1} + 1\right)^{-1} \delta,
\end{eqnarray}
where the constant $L$ is defined in Theorem \ref{lem:2d}.
Then there exists  a sequence $u^k$ of $C^2$-solutions to
\reff{th-2}, (\ref{eq:3}) such that
$\|u^k\|_{BC^2(\Pi;\R^n)} \le \delta$ for all $k$.
}

The proof will be done using induction in $k$.
To treat the base case $k=0$, let us
   construct $u^1(x,t)$.    Consider \reff{th-2}, (\ref{eq:3}) for $k=0$ and switch
to the abstract problem setting. Recall that the equivalence of both settings is proved in Section~\ref{sec:equiv}.
Since $A(x,t,0)=a(x,t)$ and  $B(x,t,0)=b(x,t)$, the homogeneous system
\reff{th-2}, (\ref{eq:3})  (or, the same, its abstract version
(\ref{unperturb}) with $f=0$)
is dichotomous by the assumption. This implies  (see \cite{Latn}) that the nonhomogeneous system  (\ref{unperturb})
has a unique bounded $L^2$-generalized solution
$u^{1}(t)$ given by
\begin{eqnarray}  \label{th-3}
u^{1}(t) = \int_{-\infty}^{\infty} G_0(t,s) f(s)ds,
\end{eqnarray}
where $U_0(t,s)=U(x,t)$ is the evolution operator  generated by  the linear system (\ref{eq:1u}), (\ref{eq:3}) and
$$G_0(t,s) = \left\{
               \begin{array}{l}
                 U_0(t,s) P(s), \ t \ge s, \\
                 U_0(t,s)(I - P(s)), \ t < s \\
               \end{array}
             \right.
$$
is the corresponding Green function  satisfying the inequality
$$\|G_0(t,s)\|_{\mathcal{L}(L^2((0,1);\R^n))} \le M e^{-\alpha|t-s|}, \ t,s \in \mathbb{R}.$$
Moreover,  we have
\begin{eqnarray} \label{th-33}
\| u^{1}(t)\|_{L^2((0,1);\R^n)} &\le& \int_{-\infty}^{\infty} \| G_0(t,s)\|_{\mathcal L(L^2((0,1);\R^n))} \|f(s)\|_{L^2((0,1);\R^n)} ds
\nonumber\\
&
\le &\frac{2 M}{\alpha}\|f\|_{BC(\Pi;\R^n)} \le \frac{2 M_1}{\alpha_1}\|f\|_{BC^2(\Pi:\R^n)}.
\end{eqnarray}

Let us show that  $u^1$ actually  has $C^2$-regularity.
With this aim, let us rewrite $u^{1}(t)$  in the form (see \cite[p. 228]{H})
\begin{eqnarray} \label{sm7}
 u^{1}(t) = U_0(t, t_0)u^1(t_0) +
 \int_{t_0}^t U_0(t,s)f(s)ds, \quad t \ge t_0.
\end{eqnarray}
Given an arbitrary $t_0\in\R$, the function $U_0(t,t_0)u^1(t_0)$ is an $L^2$-generalized solution to
the equation  (\ref{unperturb}) with $f=0$ (or, the same, to the system (\ref{eq:1u}), (\ref{eq:3}))
with the initial value $u^1(t_0).$ By Theorem \ref{lem:2d}, the function $[U_0(t,t_0)u^1(t_0)](x)$ has a $C^2$-regularity
for $t\ge t_0+3d, x \in [0,1]$.
Since the map $f:\mathbb{R} \to L^2((0,1);\R^n)$ is  differentiable,  the second summand
in \reff{sm7}, denoted by $w(t)$, is a classical solution to the abstract
equation (\ref{unperturb}) subjected to the  initial condition
 $w(t_0) = 0$
(see, e.g. \cite[p. 147]{Paz}, \cite[p. 197]{Krein}).
Due to Theorem \ref{distr-clas}, the function $w(t)$ is a classical solution of (\ref{unperturb}) if and only if it is an
$L^2$-generalized solution to the problem (\ref{eq:vg}), (\ref{eq:3}). By Theorem \ref{lem:2d}, the function $[w(t)](x)$
has a $C^2$-regularity for $t\ge t_0+3d, x \in [0,1]$.

As $t_0\in\R$ is arbitrary,
$u^{1}(x,t)$ has 
$C^2$-regularity in the whole domain $\overline\Pi$.
 Due to the inequalities \reff{apr_vv} and (\ref{th-33}),
it satisfies the following smoothing estimate:
\begin{eqnarray*} 
\| u^{1}(\cdot,t)\|_{C^2([0,1];\R^n)}& \le& L \left( \|u^1(t-3d)\|_{L^2((0,1);\R^n)} + \|f\|_{BC^2(\Pi;\R^n)}\right)
\\ &\le&
 L \left( \frac{2M_1}{\alpha_1} + 1\right) \|f\|_{BC^2(\Pi;\R^n)}.
\end{eqnarray*}
If $f$  fulfills \reff{sm7aa}, then $\|u^1\|_{BC^2(\Pi;\R^n)} \le \delta$.
As a consequence,  the linear system (\ref{th-2}), (\ref{eq:3})
for $k=1$
is exponentially dichotomous,  with the same constants $\al_1$ and $M_1.$

Assuming that Claim 1 is true for some $k\ge 1$, let us prove it for $k+1$.
Suppose that
$u^k$ is found such that
$\|u^k\|_{BC^2(\Pi;\R^n)} \le \delta.$
Then the homogeneous system (\ref{th-2})
(\ref{eq:3}) has an exponential dichotomy with the  same constants $\alpha_1$ and $M_1$.
Consider
$$u^{k+1}(t) = \int_{-\infty}^{\infty} G_k(t,s) f(s)ds,$$
where 
$$G_k(t,s) = \left\{
               \begin{array}{l}
                 U_k(t,s) P_k(s), \ t \ge s, \\
                 U_k(t,s)(I - P_k(s)), \ t < s, \\
               \end{array}
             \right.
$$
$U_k(t,s)$ is the evolution operator  generated by
the linear homogeneous system (\ref{th-2}), (\ref{eq:3}),  and $P_k$  and $I-P_k$
are the corresponding dichotomy projectors.
The  Green function  $G_k(t,s)$ satisfies the inequality
$$\|G_k(t,s)\|_{\mathcal L(L^2((0,1);\R^n))} \le M_1 e^{-\alpha_1|t-s|}, \ t,s \in \mathbb{R}.$$

Similarly to the above, we see that $u^{k+1}$ is $C^2$ smooth.
Moreover, due to (\ref{sm7aa}),  the function $u^{k+1}$ fulfills the estimate
\begin{eqnarray} \label{sm7aaa}
\|u^{k+1}\|_{BC^2(\Pi;\R^n)} \le L \left( \frac{2M_1}{\alpha_1} + 1\right) \|f\|_{BC^2(\Pi;\R^n)} \le \delta.
\end{eqnarray}
\bigskip
Claim 1 is therewith proved.

{\it Claim 2. Let
$$
\eps=\min\left\{\de L^{-1}\left( \frac{2M_1}{\alpha_1} + 1\right)^{-1}, (L^2N)^{-1}\left( \frac{2M_1}{\alpha_1} + 1\right)^{-2}
\right\}.
$$
If $\|f\|_{BC^2(\Pi;\R^n)} \le\eps$, then the sequence $\{u^k\}$ converges in   $BC^1(\overline\Pi;\R^n).$}
The difference $w^{k+1} = u^{k+1} - u^k$ belongs to $BC^2(\overline\Pi;\R^n)$ and satisfies the  system
 \begin{eqnarray}
 & & \d_ tw^{k+1} + a^k(x,t) \partial_ x w^{k+1}+ b^k(x,t) w^{k+1} = f^k(x,t),
 \label{th-5}
 \end{eqnarray}
 with the boundary conditions (\ref{eq:3}), where
  \begin{eqnarray*}
  f^k(x,t)
  = - \left(b^k(x,t) -   b^{k-1}(x,t)\right)u^k(x,t)
 - \left(a^k(x,t) -   a^{k-1}(x,t)\right)\partial_ xu^k(x,t).
 \end{eqnarray*}
 The right-hand side of (\ref{th-5}) is $C^1$-smooth in $x$ and $t$ and  satisfies the estimate
$$
  \|f^{k}\|_{BC^1(\Pi;\R^n)} \le N_1 \|u^k\|_{BC^2(\Pi;\R^n)} \|w^k\|_{BC^1(\Pi;\R^n)},
 $$
where  the constant $N_1$ depends on $A(x,t,u^k)$ and $B(x,t,u^k)$ but not on $w^k$. Additionally,
since the estimate
 (\ref{sm7aaa}) is uniform in $k$, $N_1$ can be chosen common for all  $k \in \mathbb{Z}$.

 Analogously to (\ref{th-3}) and (\ref{th-33}), $w^{k+1}: \mathbb{R} \to L^2((0,1);\R^n)$ reads
\begin{eqnarray*}  
w^{k+1}(t) = \int_{-\infty}^{\infty} G_k(t,s) f^k(s)ds,
\end{eqnarray*}
 and satisfies the estimate
 \begin{eqnarray} \label{th-33w}
 \| w^{k+1}(t)\|_{L^2((0,1);\R^n)} & \le &  \int_{-\infty}^{\infty} \| G_k(t,s)\|_{\mathcal L(L^2((0,1);\R^n))} \|f^k\|_{BC^1(\Pi;\R^n)} ds \nonumber\\
 & \le &
\frac{2 M_1 N_1}{\alpha_1}\|u^k\|_{BC^2(\Pi;\R^n)}\|w^k\|_{BC^1(\Pi;\R^n)}.
\end{eqnarray}

Now,  consider $w^{k+1}(x,t)$ as a solution to the initial-boundary value problem (\ref{th-5}),
 (\ref{eq:3}) with the initial value $w^{k+1}(t-2d)$. Using Theorem \ref{lem:2d}
and the inequalities (\ref{sm7aaa}) and (\ref{th-33w}),  we get
\begin{eqnarray*}
 \| w^{k+1}(t)\|_{BC^1([0,1];\R^n)} & \le & L \left( \|w^{k+1}(t-2d)\|_{L^2((0,1);\R^n)} + \|f^k\|_{BC^1(\Pi;\R^n)}\right)  \nonumber\\[2mm]
 & \le & L N_1 \left( \frac{2M_1}{\alpha_1} + 1\right) \|u^k\|_{BC^2(\Pi;\R^n)} \|w^k\|_{BC^1(\Pi;\R^n)}  \nonumber\\[2mm]
 & \le & L^2 N_1 \left( \frac{2M_1}{\alpha_1} + 1\right)^2 \|f\|_{BC^2(\Pi;\R^n)} \|w^k\|_{BC^1(\Pi;\R^n)}.
\end{eqnarray*}
If
\begin{eqnarray} \label{un0}
 L^2 N_1 \left( \frac{2M_1}{\alpha_1} + 1\right)^2 \|f\|_{BC^2(\Pi;\R^n)} < 1
\end{eqnarray}
then the sequence $\{ w^k \}$ tends to zero in  $BC^1(\overline\Pi;\R^n).$ Consequently, if
$\|f\|_{BC^2(\Pi;\R^n)} \le\eps$, then
 $f$ fulfills the inequalities \reff{sm7aa} and \reff{un0}, which implies that
the sequence $u^k$ converges in $BC^1(\overline\Pi;\R^n)$ to some function $u^*\in BC^1(\overline\Pi;\R^n)$.
It is a simple matter to show that the  function $u^*$ is a classical solution to the problem (\ref{eq:1}), (\ref{eq:3})
and satisfies the following estimate:
\begin{eqnarray}
 \|u^*\|_{BC^1(\Pi;\R^n)} \le L \left( \frac{2M_1}{\alpha_1} + 1\right) \|f\|_{BC^2(\Pi;\R^n)} \le \delta.
\label{*1*}
\end{eqnarray}

{\it Claim 3.  If $\|f\|_{BC^2(\Pi;\R^n)} \le\eps$,
  then the classical solution
$u^*$ to the
problem (\ref{eq:1}), (\ref{eq:3}) satisfying the bound \reff{*1*} is unique.}
On the contrary, suppose  that $\tilde u$ is a solution to the problem (\ref{eq:1}), (\ref{eq:3})
different from $u^*$,
such that $ \|\tilde u\|_{BC^1(\Pi;\R^n)} \le \delta.$ Then the linear system
$$
 \partial_t u + \tilde a(x,t) \partial_x u + \tilde b(x,t) u = 0
$$
with the boundary conditions (\ref{eq:3}), where $\tilde a(x,t) = A(x,t,\tilde u(x,t)), \ \tilde b(x,t) = B(x,t,\tilde u(x,t)),$
is exponentially dichotomous with the same constants $\alpha_1$ and $M_1.$
Clearly, the difference $\tilde w^{k+1} = \tilde u - u^{k+1}$ satisfies the  system
$$
\partial_t u + \tilde a(x,t) \partial_x u + \tilde b(x,t) u = \tilde f^{k+1}(x,t)
$$
 with the boundary conditions (\ref{eq:3}), where
  \begin{eqnarray*}
  \tilde f^{k+1}(x,t)
  = \left(b^{k}(x,t) - \tilde b(x,t)\right)u^{k+1}(x,t) +
 \left(a^{k}(x,t) - \tilde a(x,t)\right)\partial_x u^{k+1}(x,t).
 \end{eqnarray*}
Similarly to the above, the function $\tilde f^{k+1}(x,t)$ is $C^1$-smooth in $x$ and $t$ and
satisfies estimate
$$
  \|\tilde f^{k+1}\|_{BC^1(\Pi;\R^n)} \le N_1 \|u^{k+1}\|_{BC^2(\Pi;\R^n)} \|\tilde w^k\|_{BC^1(\Pi;\R^n)}.
$$
Applying the same estimates as for $w^k$, we derive the bound
\begin{eqnarray*}
& & \|\tilde w^{k+1}(t)\|_{BC^1([0,1];\R^n)}   \le
 L^2 N_1 \left( \frac{2M_1}{\alpha_1} + 1\right)^2 \|f\|_{BC^2(\Pi;\R^n)} \|\tilde w^k\|_{BC^1(\Pi;\R^n)}.
\end{eqnarray*}
Combining it with  (\ref{un0}), we get the convergence $\|\tilde w^{k}(t)\|_{BC^1(\Pi;\R^n)} \to 0$
as $k \to \infty.$ Consequently, $\tilde u(x,t) = u^*(x,t)$, a contradiction.

\subsection{Almost Periodic Solutions}
Here we prove Theorem \ref{main} $(\io\io)$, 
the almost periodic case.

 Recall that a continuous function $f: \mathbb{R} \to X$ with values in a Banach space $X$
is called a Bohr almost periodic if for
every $\varepsilon > 0$ there exists a relatively dense set of $\varepsilon$-almost
periods of $f,$ i.e., for every $\varepsilon > 0$ there exists a positive number $l$
such that every interval of length $l$ contains a number $\tau$ such that
$$\| f(t + \tau) - f(t) \|_{X} < \varepsilon \ \ {\rm for \ all} \ \ t \in \mathbb{R}.$$

As shown in Section \ref{sec:bounded}, we are done if we show that, under the assumption that the coefficients $A(x,t,v)$,
$B(x,t,v)$, and $f(x,t)$ are almost periodic in $t$,
the constructed solution $u^*(x,t)$
 is almost periodic in $t$ as well.
We use the fact that the limit of a uniformly convergent sequence of almost periodic functions is almost periodic \cite{Cord}. 
 This means that it suffices to show that the approximating sequence $\{u^k\}$  is a sequence of
  almost periodic functions.

We will also use the fact that  if a function $w(x,t)$ has bounded and continuous partial derivatives up to the second order in
 $x \in [0,1]$ and in $t \in \mathbb R$ and  is Bohr almost periodic in $t$ uniformly with respect to $x$,
 then the partial derivatives $\partial_x w(x,t)$
 and $\partial_t w(x,t)$ are also Bohr almost periodic in $t$ uniformly in $x$. 

The almost periodicity of
$\partial_t w(x,t)$ follows from \cite[Theorem 1.16]{fink} and its proof.
To prove the almost periodicity of
$\partial_x w(x,t)$, we use a similar argument. 
Let $x\in [0,1/2]$. Consider the sequence of almost periodic  functions in $t$
$$w^k(x,t) = k\left[w\left(x+1/k ,t\right) - w(x,t)\right],\quad k\ge 2$$
and prove that it tends to $\partial_x w(x,t)$ as $k \to \infty$, uniformly in $t \in \mathbb R$ and
$x\in [0,1/2]$.
Indeed,
\begin{eqnarray*}
k\left[ w\left(x + 1/k, t\right) - w(x,t)\right] -  \partial_x w(x,t) =
k\int_0^{1/k}\left(  \partial_x w(x + \xi,t) - \partial_x w(x,t)\right) d\xi.
 \end{eqnarray*}
 Since $w\in BC^2(\overline\Pi;\R^n)$, then for any 
 $\varepsilon > 0$ there exists $k_0\in\N$ such that
 $$
\left|\partial_x w(x + \xi,t) - \partial_x w(x,t)\right| < \varepsilon
$$
for all $\xi\in[0,1/k_0]$,  uniformly in $x\in [0,1/2]$ and $t\in\R.$
This means that the sequence of almost periodic functions $\{w^k(x,t)\}$ tends to  $\partial_x w(x,t) $
uniformly in $x\in [0,1/2]$ and $t \in \mathbb R$. Hence, the function $\partial_x w(x,t)$
is almost periodic in $t$ uniformly in $x\in [0,1/2].$

For  $x\in [1/2,1]$ we proceed similarly considering the sequence
$$w^k(x,t) = k\left[w\left(x-1/k ,t\right) - w(x,t)\right],\quad k\ge 2.$$
Summarizing, this shows the  almost periodicity of
$\partial_x w(x,t)$.

Turning back to the almost periodicity of $u^*$,  first recall that $u^0\equiv 0$.
Assuming that the solution $u^k(x,t)$ to (\ref{th-2}) is almost periodic in $t$ uniformly in $x$, let us prove that $u^{k+1}(x,t)$ is almost periodic also.
Fix an arbitrary $\varepsilon > 0.$ Let $h$ be an $\varepsilon$-almost period of almost
periodic in $t$  functions $f(x,t),$ $a^k(x,t) = A(x,t,u^k(x,t)),$ $b^k(x,t) = B(x,t,u^k(x,t))$ as well as
their derivatives in $x$ and $t$.
Then the differences $\tilde a^k(x,t) = a^k(x,t + h) - a^k(x,t)$ and $\tilde b^k(x,t) = b^k(x,t + h) - b^k(x,t)$ satisfy
the inequalities
\begin{eqnarray} \label{tilde}
 \|\tilde a^k\|_{BC^1(\Pi;\M_n)} \le \varepsilon, \  \|\tilde b^k\|_{BC^1(\Pi;\M_n)} \le \varepsilon.
\end{eqnarray}
We are done if we prove that $h$ is an almost period of the function $u^{k+1}(x,t).$

It is known that the functions $u^{k+1}(x,t)$ and $u^{k+1}(x,t+h)$ are unique bounded  solutions, respectively,
 to the  
system
 (\ref{th-2}) with the boundary conditions (\ref{eq:3}) and to the  system
$$
 \partial_t u + a^k(x,t+h)\partial_x u  + b^k(x,t+h)u = f(x,t+h)
$$
with the boundary conditions (\ref{eq:3}).
The difference $z^{k+1}(x,t) = u^{k+1}(x,t) - u^{k+1}(x,t+h)$ satisfies the system
$$
\partial_t z^{k+1} + a^k(x,t) \partial_x z^{k+1} + b^k(x,t) z^{k+1} = g^k(x,t)
 $$
 subjected to (\ref{eq:3}), where
  \begin{eqnarray*}
  & & g^k(x,t)
  = - \left(b^k(x,t) -  b^k(x,t+h)\right)u^{k+1}(x,t+h) \\
 & & - \left(a^k(x,t) -   a^k(x,t+h)\right)\partial_x u^{k+1}(x,t+h)
 +  f(x,t) - f(x,t+h).
 \end{eqnarray*}
 The function $g^k(x,t)$ is $C^1$-smooth in $x$ and $t$ and, due to \reff{sm7aaa} and \reff{tilde},
  satisfies the bound
  \begin{eqnarray}
  \|g^{k}\|_{BC^1(\overline\Pi;\R^n)} \le \varepsilon(\delta+1).
  \label{th-6h}
 \end{eqnarray}
 Analogously to (\ref{th-3}) and (\ref{th-33}), $z^{k+1}: \mathbb{R} \to L^2((0,1);\R^n)$ is defined by the formula
\begin{eqnarray*}  
z^{k+1}(t) = \int_{-\infty}^{\infty} G_k(t,s) g^k(s)ds.
\end{eqnarray*}
 Moreover, the following  estimate is true:
 \begin{eqnarray} \label{th-33wh}
\| z^{k+1}(t)\|_{L^2((0,1);\R^n)} \le \int_{-\infty}^{\infty} \| G_k(t,s)\|_{\mathcal L(L^2((0,1);\R^n))} \|g^k\|_{BC^1(\Pi;\R^n)} ds \le
\frac{2 M_1}{\alpha_1}(\delta+1)\varepsilon,
\end{eqnarray}
the bound being uniform in $t\in\R$. We combine  the representation
 \begin{eqnarray*}  
 z^{k+1}(t) = U_k(t,t-2d)z^{k+1}(t - 2d) + \int_{t-2d}^t U_k(t,s) g^k(s) ds
  \end{eqnarray*}
with Lemma \ref{lem:2d} and the inequalities (\ref{th-6h}) and (\ref{th-33wh}). This yields
the desired estimate
\begin{eqnarray*}
& & \| z^{k+1}(t)\|_{C^1([0,1];\R^n)}  \le L \left( \|z^{k+1}(t-2d)\|_{L^2((0,1);\R^n)} + \|g^k\|_{BC^1(\Pi;\R^n)}\right) \le \nonumber\\[2mm]
& & \le \varepsilon L (\delta+1)\left( \frac{2M_1}{\alpha_1} + 1\right) = L_1 \varepsilon
\end{eqnarray*}
 or, the same,
  \begin{eqnarray*}
\|u^{k+1}(t) - u^{k+1}(t+h)\|_{C^1([0,1];\R^n)} \le  L_1\varepsilon,
 \end{eqnarray*}
the constant $L_1$ being independent of $\varepsilon$ and $t\in\R$. This finishes the proof of the almost periodicity
of $u^{k+1}(x,t).$

\subsection{Periodic Solutions}
If the coefficients $A(x,t,v)$,
$B(x,t,v)$, and $f(x,t)$ are $T$-periodic in $t$,
then each constructed iteration $u^k$ is in fact 
a unique solution 
to a linear dichotomous problem with $T$-periodic in $t$
coefficients. This yields the  $T$-periodicity in $t$ of $u^k$
and, hence the  $T$-periodicity in $t$ of the limit
function $u^*$.
 
The proof of Theorem \ref{main} is complete.
 
\section*{Acknowledgments}
 This work  was supported by the 
VolkswagenStiftung Project ``Modeling, Analysis, and Approximation Theory toward Applications in 
Tomography and Inverse Problems''.

\end{document}